\documentclass[reqno,11pt]{amsart}
\usepackage{amsmath,amssymb,amsthm,graphicx,mathrsfs,url}
\usepackage{wrapfig}
\usepackage{enumitem}
\usepackage{mathtools}
\usepackage{xparse}
\usepackage[usenames,dvipsnames]{xcolor}
\usepackage[colorlinks=true,linkcolor=Red,citecolor=Green]{hyperref}
\usepackage[super]{nth}
\usepackage[open, openlevel=2, depth=3, atend]{bookmark}
\hypersetup{pdfstartview=XYZ}
\usepackage[font=footnotesize]{caption}
\usepackage{a4wide}
\theoremstyle{plain}
\newtheorem{theo}{Theorem}
\newtheorem{prop}{Proposition}[section]
\newtheorem{lemm}[prop]{Lemma}
\newtheorem{corr}[prop]{Corollary}

\theoremstyle{definition}
\newtheorem{defi}[prop]{Definition}

\theoremstyle{remark}
\newtheorem{rem}{Remark}

\numberwithin{equation}{section}

\DeclareMathOperator{\supp}{supp}

\DeclareMathOperator{\Op}{Op}

\def\d{\,\mathrm{d}}
\def\C{\mathbb{C}}

\def\eps{\varepsilon}

\def\la{\left\vert}
\def\lA{\left\Vert}

\def\lB{\left\{}
\def\lp{\left(}
\def\ls{\left\langle}

\def\N{\mathrm{N}}

\def\R{\mathbb{R}}
\def\ra{\right\vert}
\def\rA{\right\Vert}

\def\rB{\right\}}
\def\rp{\right)}
\def\rs{\right\rangle}

\renewcommand{\Re}{\mathrm{Re}}
\renewcommand{\Im}{\mathrm{Im}}

\newcommand{\mc}{\mathcal}
\newcommand{\rr}{\mathbb{R}}
\newcommand{\nn}{\mathbb{N}}

\newcommand{\pl}{\partial}

\newcommand{\bbar}{\overline}

\newcommand{\cjd}{\rangle}
\newcommand{\cjg}{\langle}

\NewDocumentCommand{\ceil}{s O{} m}{%
  \IfBooleanTF{#1} % starred
    {\left\lceil#3\right\rceil} % \ceil*[..]{..}
    {#2\lceil#3#2\rceil} % \ceil[..]{..}
}

\let\Im=\Imag

\let\Re=\Real

\DeclareMathOperator{\WF}{WF}

\makeatletter\let\@wraptoccontribs\wraptoccontribs\makeatother

\title[A paradifferential approach for hyperbolic dynamical systems]{A paradifferential approach for hyperbolic dynamical systems and applications}
\author{Colin Guillarmou}
\email{colin.guillarmou@math.u-psud.fr}
\address{Laboratoire de Math\'ematiques d'Orsay, Univ. Paris-Sud, CNRS, Universit\'e Paris-Saclay, 91405 Orsay, France}
\author{Thibault de Poyferr\'e}
\email{tdepoyferre@msri.org}
\address{MSRI Berkeley, 17 Gauss Way, Berkeley, CA 94720}

\contrib[with an appendix by]{Yannick Guedes Bonthonneau}
\email{bonthonneau@math.univ-paris13.fr}
\address{LAGA, Institut Galil\'ee, 99 avenue Jean Baptiste cl\'ement, 93430 Villetaneuse, France.}

\begin{document}
\maketitle

\begin{abstract}
We develop a paradifferential approach for studying non-smooth hyperbolic dynamics on manifolds and related non-linear PDE from a microlocal point of view. As an application, we describe the 
microlocal regularity, i.e the $H^s$ wave-front set for all $s$, of the unstable bundle $E_u$ for an Anosov flow. We also recover rigidity results of Hurder-Katok and Hasselblatt in the Sobolev class rather than H\"older: there is $s_0>0$ such that if $E_u$ has $H^s$ regularity for $s>s_0$ then it is smooth (with $s_0=2$ for volume preserving $3$-dimensional Anosov flows). It is also shown in the Appendix that it can be applied to deal with non-smooth flows and potentials. This work could serve as a toolbox for other applications.
\end{abstract}

\section{Introduction}
Consider a smooth compact manifold $M$ (equipped with a fixed metric $g$) and let $X$ be a smooth non-vanishing 
vector field generating an Anosov flow $\varphi_t$. This means that the tangent bundle $TM$ has a $d\varphi_t$ invariant decomposition  
\[ TM=\rr X \oplus E_u\oplus E_s\]
such that there is $C>0$ and $\nu>0$ so that for each $t\geq 0$
\begin{equation}\label{Anosovdef}
\begin{gathered}
\forall x\in M, \forall w\in E_u(x),\quad \| d\varphi_{-t}(x).w\|\leq Ce^{-\nu t}\|w\|\\ 
\forall x\in M, \forall w\in E_s(x),\quad \| d\varphi_t(x).w\|\leq Ce^{-\nu t}\|w\|.
\end{gathered}\end{equation}
The stable and unstable bundles $E_s$ and $E_u$ are only H\"older continuous \cite{HPS70}. A sharp regularity statement on the H\"older exponent of $E_u/E_s$ has been obtained by Hasselblatt \cite{Hasselblatt-regularityI, Hasselblatt-regularityII} in terms of bunching of the expansions/contraction exponents of $d\varphi_t$ on $E_u/E_s$; we refer to these articles for a history and references of the analysis of the regularity. 

A particularly interesting case is when $X$ is the generator of the geodesic flow on a negatively curved manifold. That class fits into the more general class of \emph{contact Anosov flows}, defined by the requirement that  the Anosov form~$\alpha$, defined by
\[\alpha(X)=1, \quad \alpha(E_u\oplus E_s)=0\]
is a contact form (which means that it is smooth and $d\alpha|_{\ker \alpha}$ is non-degenerate and $i_Xd\alpha=0$).
A striking rigidity result of Hurder-Katok \cite{Hurder-Katok} is that for contact Anosov flow 
in dimension $3$ (for example the geodesic flow of a negatively curved surface), the stable and unstable bundles $E_s$ and $E_u$ are always\footnote{Here, ${\rm Zyg}$ is the Zygmund modulus of continuity.}~$C^{1,O({\rm Zyg})}$, and if they are~$C^{1,o({\rm Zyg})}$ they are necessarily $C^\infty$. In particular $C^2$ regularity implies $C^\infty$ regularity for $E_u$ and $E_s$.
On the other hand Ghys \cite{Ghys87} proved  that if $E_u$ and $E_s$  are~$C^\infty$, then the flow is conjugate to the geodesic flow on the circle bundle of a hyperbolic surface, modulo finite covering and quotients, and smooth time change. To conclude, $E_s\in C^2$ and $E_u\in C^2$ implies that the flow is essentially a homogeneous dynamical system. The corresponding rigidity result has been proved\footnote{It is claimed in the paper that a $C^k$ regularity for $k$ large enough depending only on the dimension is sufficient.} by Benoist-Foulon-Labourie \cite{Benoist1992} in higher dimension provided $E_s, E_u\in C^\infty$. 
Hasselblatt \cite{Hasselblatt1992PAMS} also proved ridigity results in any dimension, showing that if $E_u\in C^k$ for $k$ larger than a constant depending on the bunching of the contraction/expansion exponents appearing in \eqref{Anosovdef}, then the bundles are smooth. It is a folklore conjecture that $C^2$-regularity of $E_u,E_s$ for a contact Anosov flow should imply $C^\infty$ regularity and then, using \cite{Benoist1992}, that the flow is smoothly conjugate to a homogeneous dynamical system.   

In this work, one of the aspects we propose to study is the microlocal regularity of $E_u$ (and $E_s$) through the introduction of paradifferential methods.  By microlocal regularity we mean its regularity in phase space, i.e. $T^*M$, which can be encoded in the notion of $H^s$ wave-front set introduced by H\"ormander \cite[Chapter 8]{Hoe03}.  Roughly speaking, it encodes where a function
is singular (the position $x$) but also in which direction it is singular (the momentum or Fourier variable $\xi$ decay). For example, the regularity of the unstable bundle for a Riemannian negatively curved surface $\Sigma$ is encoded by a function $r\in C^\gamma(S\Sigma)$ satisfying a Ricatti equation 
\[Xr+r^2+K=0\]
where $K$ is the Gauss curvature. It is non-linear, and thus not very suitable for microlocal study 
using the standard pseudo-differential operator calculus. We thus employ the paradifferential calculus of Bony \cite{Bony1981} to approach this problem, and in particular we show that the formalism of radial estimates used recently by Dyatlov-Zworski \cite{DZ16a} for hyperbolic dynamics 
fits well in that picture. They allow to describe the sharp Sobolev regularity of the unstable bundle, its wave-front set and to recover the rigidity property mentionned above but in a Sobolev class: if the bundle $E_s$ is $H^s$ for $s$ large enough (depending on the maximal/minimal expansions rates of the flow), then $E_u$ is smooth provided the flow is smooth.

More precisely this method allows to show (see Theorem \ref{Thregulrigid1} and \ref{Thregulrigid2}) that 
\[{\rm WF}(E_u)\subset E_u^*,\quad  {\rm WF}(E_s)\subset E_s^*\]
if $E_u^*,E_s^*\subset TM$ are defined by $E_u^*(E_u \oplus \R X)=0$ and $E_s^*(E_s \oplus \R X)=0$. This fact could be useful for the microlocal approach of Ruelle resonances, in particular in the understanding of applying unstable derivatives to irregular distributions.
We also recover Sobolev regularity on $E_u,E_s$ and the rigidity results for this regularity (if it is $H^s$ for $s>0$ large enough then it is as smooth as the flow) using this PDE approach.
These regularity/rigidity properties were studied for example by Hasselblatt, Hurder-Katok, De la LLave (\cite{Hurder-Katok,Hasselblatt1992PAMS, Hasselblatt-regularityI,DLL2001}).

The method we employ is based on propagation estimates: we use Bony's propagation of singularity for paradifferential operators \cite{Bony1981} and we show a version of the radial point estimates for paradifferential operators. This could have other applications. 
Radial point estimates in the classical cases were introduced by Melrose \cite{Melrose1994} and developed also by Vasy \cite{Va13}, Dyatlov-Zworski \cite{DZ16a,DZ19} in the particular settings of Anosov flows, see also Dyatlov-Guillarmou \cite{DG16} for more general cases (Axiom A flows), Wang \cite{Wang-Jian} in Besov spaces and Guedes Bonthonneau-Lefeuvre \cite{Bonthonneau-Lefeuvre} in the H\"older Zygmund classes.\\

Our method also allows to understand the sharp Sobolev regularity and rigidity phenomena for solutions $U$ of a general Riccati equation of the form 
\[ \mc{L}_XU + Q(x,U)=0\]
where $U$ is a H\"older sections of an ${\rm End}(E)$ for some smooth bundle $E$ on which the flow $\varphi_t$ admits a (linear) lifted action $\tilde{\varphi}_t$, $Q$ is a quadratic polynomial in $U$ with smooth coefficients in $x$, $\mc{L}_X$ is the Lie derivative in the direction of an Anosov vector field $X$. If the regularity of $U$ is larger than a certain threshold $s$ depending on $\nu_{s/u}^{\min/\max}$ and on $\pl_UQ(\cdot,U)$, then it is smooth.
More generally, the method applied to general nonlinarities, i.e. equations of the form 
$\mc{L}_XU+F(x,U)=0$ for some smooth functional $F:M\times E\to \R$. See Remark \ref{Remarquegenerale}.\\

In turn, paradifferential methods are well-suited to study PDEs with non-smooth coefficients, and we also show in Section \ref{nonsmoothpotentials} for H\"older potentials (Proposition \ref{nonsmoothpotential}), and more generally in the Appendix by Guedes Bonthonneau for H\"older flows and potentials (Theorem \ref{nonsmoothflows}), that this provides a good microlocal method for analyzing the resolvent $(X+V+\lambda)^{-1}$ of non-smooth Anosov vector field $X$ and non-smooth potentials $V$. In particular, combining the paradifferential calculus with the microlocal approach of Faure-Sj\"ostrand \cite{FS11}, one recovers the fact that a theory of Ruelle resonances can be done also for H\"older flows and potentials in a spectral strip of finite size, as was shown by Butterley-Liverani \cite{BL07} and Adam-Baladi \cite{Adam-Baladi} for Anosov flows, or by Blank-Keller-Liverani \cite{BKL02}, Baladi-Tsujii \cite{BT07} and Gou\"ezel-Liverani \cite{GL08} for hyperbolic diffeomorphisms. 
The case of non-smooth potential is quite natural as the geometric potentials given by the unstable or stable Jacobians are only H\"older, the proof for a smooth flow is explained in Section \ref{nonsmoothpotentials} and the general case for a H\"older Anosov flow is explained by Guedes Bonthonneau in the Appendix using the construction of escape function for non-smooth flows in \cite{Bon18}.

We hope more generally that this paradifferential approach could be applied more systematically for non-smooth dynamics and related non-linear equations.\\

\textbf{Acknowledgements.} This project has received funding from the European Research Council (ERC) under the European Union’s Horizon 2020 research and innovation programme (grant agreement No. 725967). This material is based upon work supported by the National Science Foundation under Grant No. DMS-1440140 while C.Guillarmou was in residence at the Mathematical Sciences Research Institute in Berkeley, California, during the Fall 2019 semester.
This material is based upon work supported by the National Science Foundation under Grant No. DMS-1928930 while the T. De Poyferr\'e participated in a program hosted by the Mathematical Sciences Research Institute in Berkeley, California, during the spring 2021 semester.

\section{Paradifferential and $S_{1,1}$ Calculi}	\label{Para}
Paradifferential calculus is a way of blending nonlinear harmonic analysis with microlocal techniques. From the pseudodifferential viewpoint, it corresponds to quantization of symbols belonging to H\"ormander's forbidden class~$S_{1,1}$, but satisfying some conditions salvaging the existence of a calculus. We start with a review of this calculus on manifolds, since such a reference does not seem to be available.
The classical references for the Euclidean case are the original papers of Bony ~\cite{Bony1981}, Meyer \cite{Meyer81}, Bourdaud \cite{Bourdaud1988}, and H\"ormander \cite{Hormander88,Hormander-nonlinear}. We also refer to the book of Taylor \cite{Taylor_nonlinear_book} that brings a discussion on manifold. 
In what follows, $M$ will be a compact, smooth manifold of dimension~$n$, and we equip it with a smooth Riemannian metric $g_0$ that allows to put norms on the fibers of the tangent and cotangent bundles $TM$ and $T^*M$, as well as a smooth Riemannian measure ${\rm dv}_g$ on $M$. The Sobolev spaces $H^s(M)$ can then be defined as $H^s(M):=\{u\in C^{-\infty}(M)\,|\, (1+\Delta_{g_0})^{s/2}u\in L^2(M)\}$, or equivalently $u\in H^s(M)$ if and only if $u|_{U_j}\in H^s(U_j)$ for all $j$ where $(U_j)_j$ is a covering by charts diffeomorphic to the unit ball $B(0,1)\subset \R^n$. The H\"older space $C^{\alpha}(M)$ is simply $C^k(M)$ is $\alpha=k\in \N$, and
 consists more generally of functions $f$ whose $k:=[\alpha]$ derivative is $(\alpha-k)$-H\"older in charts, i.e $\sup_{x\not=y}|f(x)-f(y)|/|x-y|^{\alpha}<\infty$. Finally, $\mc{F}$ will denote the usual Fourier transform.

\subsection{H\"ormander $\tilde{\Psi}_{1,1}$ Calculus.}

\begin{defi}
For an order~$m\in\R$ and $\rho,\delta\in [0,1]$, the class $S^m_{\rho,\delta}(M)$ consists in the space of smooth functions~$a(x,\xi)$ on the cotangent bundle~$T^*M$, satisfying estimates of the form
\[ \forall\alpha,\beta, \exists C_{\alpha,\beta}>0, \forall (x,\xi)\in T^*M,\quad \la \partial_x^\alpha\partial_\xi^\beta a(x,\xi)\ra\leq C_{\alpha,\beta}\ls\xi\rs^{m+\delta\la\alpha\ra-\rho\la\beta\ra}<\infty,\]
in local charts on $M$, and for any multi-indices~$\alpha$ and~$\beta$. Here $\cjg\xi\cjd:=(1+|\xi|^2)^{1/2}$.
\end{defi}

The set~$S_{1,1}(M)=\cup_{m\in\R}S^m_{1,1}(M)$ is then a filtered Fréchet $*$-algebra, with seminorms given by the best~$C_{\alpha,\beta}$ in the estimates above. It is invariant by changes of coordinates.

It contains the symbols of differential operators. The residual symbols~$a\in S^{-\infty}(M)=\cap_{m\in\R}S^m_{1,1}(M)$ are those smooth functions which are decaying faster than any $\cjg \xi\cjd^{-N}$ for $N>0$ in the fibers of~$T^*M$.  

One can now define a quantization ${\rm Op}$  on $M$ by using a partition of unity $(\psi_j)_j$ associated to a finite family of local charts $U_j\simeq B(0,\eps)\subset \R^n$, some functions $\tilde{\psi}_j\in C_c^\infty(U_j)$ so that $\tilde{\psi}_j=1$ on ${\rm supp}(\psi_j)$, 
and using the left quantization in charts 
\[ {\rm Op}(a)f(x)=\sum_{j}\frac{1}{(2\pi)^n}\int_{\R^n}\int_{U_j} e^{i(x-y).\xi}\psi_j(x)a(x,\xi)\tilde{\psi}_j(y)f(y)dy d\xi.\]
We define the classes $\Psi^{m}_{1,1}(M)$ to be the set of operators $A$ mapping $C^{\infty}(M)$ to $C^{-\infty}(M)$ continuously so that $A={\rm Op}(a)+S$ for some $a\in S^m_{1,1}(M)$ and 
$S$ a smoothing operator, i.e. an operator with smooth Schwartz kernel. The resulting operators are invariant by change of variables, as we shall see below.
The class $\Psi^{m}_{1,1}(M)$ is not very convenient since there is no good composition law, it is not bounded on $H^{s}(M)$ for $s\leq 0$ and it is not stable when taking adjoints. 
However, there is a notion of principal symbol. To see that, recall that the formula for the action of coordinate changes on the symbol of a pseudodifferential operators in $\R^n$
does not involve spatial differentiation of the symbol: 
\begin{lemm}\label{changeofcoord}
 For~$a\in S^m_{1,1}(\R^{n})$ compactly supported in $x$, $\chi\in C_c^\infty(\R^n)$ and~$\phi$ a smooth diffeomorphism of~$\R^n$ equal to the identity outside a compact set,
 \[(\phi^{-1})^* a(x,D_x)\chi(x)\phi^*=a_\phi(x,D_x) \chi\circ \phi^{-1},\]
 where $a_\phi\in S^{m}_{1,1}(\R^n)$ is such that
 \[a_\phi(\phi(x),(d\phi(x)^{-1})^T\xi)-a(x,\xi)\in  S^{m-1}_{1,1}(\R^n).\]
\end{lemm}
\begin{proof} We can use the formula for the change of variable for a left-quantized symbol given by 
\cite[Theorem 4.1]{Shu01}: if $A_1=(\phi^{-1})^* a(x,D_x)\chi(x)\phi^*$
\[ A_1u(x)=\int e^{i(x-y)\xi}a(\phi^{-1}(x),\psi(x,y)\xi)\chi(\phi^{-1}(y))|\det(\psi(x,y))|\det d\phi^{-1}(y)|u(y)dyd\xi\]
for some smooth map $\psi(x,y)\in {\rm GL}_n(\R)$ so that $(\phi^{-1}(x)-\phi^{-1}(y)).\psi(x,y)\xi=(x-y).\xi$. As in the proof of \cite[Theorem 3.1]{Shu01}, this can be rewritten under the form 
$A_1u(x)=\int_{\R^n}e^{i(x-y).\xi}a_\phi(x,\xi)\chi(\phi^{-1}(y))u(y)dyd\xi$ where 
\[a_\phi(x,\xi)=\int\int e^{i(x-y).(\eta-\xi)}a(\phi^{-1}(x),\psi(x,y)\eta)|\det(\psi(x,y))|\det d\phi^{-1}(y)|d\eta dy\]
and \cite[Theorem 4.1]{Shu01}\footnote{Shubin does not technically considers the case of $S^m_{1,1}(\R^n)$ but the argument readily applies to that case as well.} shows that this is in $S^{m}_{1,1}(\R^n)$ with an expansion as $|\eta|\to \infty$
\[ a_\phi(\phi(x),\eta)\sim \Big(\sum_{\alpha}\frac{1}{\alpha!}\pl_{\xi}^\alpha a(x,d\phi(x)^T\eta)\cdot D_z^\alpha e^{i\phi''_x(z).\eta}|_{z=x}\Big)\chi(\phi^{-1}(x))\] 
with $\phi''_x(z)=\phi(z)-\phi(x)-d\phi(x)(z-x)$. The asymptotic expansion makes sense since each 
$\xi$ derivative improves the $\eta$ decay by one order and $\phi''_x(z)$ vanishes to second order at $z=x$.
\end{proof}
As a direct consequence of this formula, we can define a principal symbol map:

\begin{defi}\label{def-principal-symb}
 On a compact manifold~$M$, the principal symbol map
 \[ \sigma: \Psi_{1,1}^m(M)\rightarrow  S^m_{1,1}(M)/ S^{m-1}_{1,1}(M)\]
 is well defined, with kernel $\Psi_{1,1}^{m-1}(M)$.
\end{defi}
However, without additional restrictions on the symbols, this map is useless as it is not an algebra homomorphism.
\\

 H\"ormander \cite{Hormander88} introduced a subclass 
$\tilde{\Psi}_{1,1}(\R^n)$ of $\Psi_{1,1}(\R^n)$ that is stable by composition and adjoint, and is bounded on $H^s(\R^n)$ for all $s\in \R$. Let us recall its definition and properties, and exlain how this extends to a compact manifold. 
 
Let  $\chi\in C^\infty(\R^{2n})$ 
satisfying $\chi(\eta,\xi)=1$ in $\{|\eta|\leq |\xi|/2, |\xi|\geq 2\}$ and $\supp\chi\subset \{|\eta|<|\xi|, |\xi|\geq 1\}$.
We say that $a\in \tilde{S}^{m}_{1,1}(\R^n)$ if for all $\eps>0$ small, the function
\begin{equation}\label{aeps}
 a_\eps(x,\xi):=\frac{1}{(2\pi)^n}\int\int
e^{i(x-y).\eta}\chi(\xi+\eta,\eps \xi)a(y,\xi)dyd\eta=\mc{F}^{-1}_{\eta\to x}(\chi(\xi+\eta,\eps \xi)\mc{F}(a(\cdot,\xi))(\eta))
\end{equation} 
satisfies the bounds for all $\alpha,\beta, N$ 
\[ |\pl_{x}^\alpha\pl_{\xi}^\beta a_\eps(x,\xi)|\leq C_{\alpha \beta N}\eps^{N}(1+|\xi|)^{m+|\alpha|-|\beta|}.\]
In other words, the microlocalisation of $a(x,\xi)$ (in $x$) in small cones near the twisted diagonal $\xi+\eta=0$ enjoys some decay to all order in the topology of $S^{m}_{1,1}(\R^n)$.
We define $\tilde{\Psi}^{m}_{1,1}(\R^n):={\rm Op}(\tilde{S}^m_{1,1}(\R^n))$ where ${\rm Op}$ is the usual left quantization on $\R^n$. 
\begin{prop}{\cite[Theorem 3.6 and Theorem 4.2]{Hormander88}, \cite[Theor\`eme 3]{Bourdaud1988}}\label{prop:hormander}
An operator $A\in \Psi^{m}_{1,1}(\R^n)$ belongs to $\tilde{\Psi}^{m}_{1,1}(\R^n)$ if and only if 
$A^*\in \Psi^{m}_{1,1}(\R^n)$. An operator $A\in \Psi^{m}_{1,1}(\R^n)$ belongs to $\tilde{\Psi}^{m}_{1,1}(\R^n)$ if and only if $A$ is bounded as a map $H^{s+m}(\R^n)\to H^s(\R^n)$ for all $s\in \R$.
\end{prop}
For this Proposition, Bourdaud \cite{Bourdaud1988} proved that the largest subalgebra of those operators $A\in \mc{L}(L^2)$ contained $\Psi_{1,1}^0(\R^n)$ is the space of those $A\in\Psi_{1,1}^0(\R^n)$ so that $A^*\in \Psi_{1,1}^0(\R^n)$, and H\"ormander \cite{Hormander88} gave the $\tilde{\Psi}_{1,1}(\R^n)$ characterization and the necessary conditions for boundedness on $H^s(\R^n)$ for all $s\in\R$. 

Since for a smooth diffeomorphism $\phi$ equal to ${\rm Id}$ outside a compact set, we have $\phi^*:H^s(\R^n)\to H^s(\R^n)$ for all $s\in\R$, we deduce from Proposition \ref{prop:hormander} that for $a\in \tilde{S}^m_{1,1}(\R^{n})$ compactly supported in $x$, and~$\phi$ a diffeomorphism of~$\R^n$ equal to the identity outside a compact set,
 \[(\phi^{-1})^* a(x,D_x)\phi^*\in \tilde{\Psi}^m_{1,1}(\R^n).\]
This implies that one can define the same class on a compact manifold $M$:
\begin{defi}
We say that $A\in \tilde{\Psi}^m_{1,1}(M)$ if $A\in \Psi^m_{1,1}(M)$ and for each $j$, 
$\tilde{\psi}_jA\psi_j\in \tilde{\Psi}^{m}_{1,1}(\R^n)$, i.e.
its symbol in the charts $U_j$ is the restriction of elements in $\tilde{S}^{m}_{1,1}(\R^n)$ to $U_j$.
\end{defi}
Due to the invariance of $\tilde{\Psi}^{m}_{1,1}(\R^n)$ under diffeomorphism, 
we see that this definition is independent of the choice of charts. 

The wave-front set ${\rm WF}(A)$ of such an operator is then defined as in the classical case as the conic set of points 
$(x_0,\xi_0)\in T^*M\setminus \{0\}$ which do not admit a conic neighborhood in $T^*M$ where 
the symbol $a(x,\xi)$ in charts is an $\mc{O}_{C^N}(\cjg \xi\cjd^{-N})$ for all $N>0$; its properties are identical the those of the classical class $\Psi_{1,0}(M)$.\\

It has been known since Stein that~$\mathrm{Op}(S^m_{1,1}(\R^n))$ maps~$H^{s+m}(\R^n)$ to~$H^s(\R^n)$ for any~$s>0$, but there are counter-examples in the case~$s=0$. However, the class $\tilde{\Psi}_{1,1}(M)$ has all the desired properties. From the boundedness and composition properties on $\R^n$, we directly obtain the following results on $M$: 
\begin{prop}\label{boundedness}
The following properties hold: 
\begin{itemize}
\item \cite[Theorem 1]{Bourdaud1988}, following Stein's proof.  
Let $A\in \Psi^{m}_{1,1}(M)$, then $A: H^{s+m}(M)\to H^{s}(M)$ is bounded for all $s>0$ and 
$A: C^{s+m}(M)\to C^{s}(M)$ if $s,s+m>0$ are non-integer.
\item  \cite[Theorem 3.6]{Hormander88} If~$A\in\tilde\Psi^m_{1,1}(M)$, then~$A$ maps~$H^{s+m}(M)$ to~$H^s(M)$ for any~$s\in\R$.
\item \cite[Theorem 5.2]{Hormander88} If $A\in \Psi^\mu_{1,1}(M)$, $B\in \tilde\Psi^m_{1,1}(M)$ and $C\in \tilde{\Psi}^\mu_{1,1}(M)$ then $AB\in \Psi_{1,1}^{m+\mu}(M)$ and $BC\in \tilde{\Psi}_{1,1}^{m+\mu}(M)$.
\item \cite[Theorem 4.2]{Hormander88} If $A\in \tilde\Psi^m_{1,1}(M)$ then $A^*\in \tilde\Psi^m_{1,1}(M)$.
 \end{itemize}
\end{prop}
In fact, operators in~$\tilde\Psi^0_{1,1}(M)$ are Calder\'on-Zygmund and thus act on~$L^p(M)$ for~$1<p<\infty$, Besov, Hardy and BMO spaces; see the book \cite{Taylor_nonlinear_book} for very general results in this direction.\\

Although~$\tilde\Psi_{1,1}(M)$ is an algebra, it is still too big for a calculus to exist, in the sense that the principal symbol is not a homomorphism. 
Thus we need to introduce smaller subspaces of symbols:
\begin{defi}
 For~$r>0$ and~$m\in\R$, the space of $r$-regular symbols~$\prescript{r}{}{S}^m_{1,1}(M)$, is the space of~$a\in S^m_{1,1}(M)$ such that in local charts
 \[\forall\alpha,\beta, \exists C_{\alpha,\beta}>0, \forall (x,\xi)\in T^*M, \quad \la \partial_x^\alpha\partial_\xi^\beta a(x,\xi)\ra\leq C_{\alpha,\beta}\ls\xi\rs^{m+(\la\alpha\ra-r)_+-\la\beta\ra},\]
 with~$f_+=\max(0,f)$. We also define $\prescript{r}{}{\tilde S}^m_{1,1}(M):=\prescript{r}{}{S}^m_{1,1}(M)\cap \tilde{S}^{m}_{1,1}(M)$. 
 \end{defi}
Note that for $r\geq r'\geq 0$, we have $\prescript{r-r'}{}S^{m-r'}_{1,1}(M)
\subset \prescript{r}{}S^{m}_{1,1}(M)$ and that for $a\in \prescript{r}{}{S}^m_{1,1}(M)$, $b\in \prescript{r}{}{S}^{m'}_{1,1}(M)$ one has 
\begin{equation}\label{propertiesclassrPsi}
ab \in \prescript{r}{}{S}^{m+m'}_{1,1}(M), \quad \pl_\xi a\in \prescript{r}{}{S}^{m-1}_{1,1}(M), \quad \pl_x a\in \prescript{r-1}{}{S}^{m}_{1,1}(M). 
\end{equation}
If $S^m(M)=S^{m}_{1,0}(M)$ denotes the standard class of symbols satisfying  
$|\pl_x^\alpha\pl_{\xi}^\beta a(x,\xi)|\leq C_{\alpha\beta}(1+|\xi|)^{m-|\beta|}$, and if ~$0\leq r\leq r'$, we also have
\[S^m(M)=\prescript{\infty}{}{\tilde S}^m_{1,1}(M)\subset\prescript{r'}{}{\tilde S}^m_{1,1}(M)\subset\prescript{r}{}{\tilde S}^m_{1,1}(M)\subset\prescript{0}{}{\tilde S}^m_{1,1}(M)=\tilde S^m_{1,1}(M).\]
The residual symbols are those in~$\prescript{\infty}{}{\tilde S}^{-\infty}_{1,1}(M)=S^{-\infty}(M)$.
We define $\prescript{r}{}{\Psi}_{1,1}^m(M)$ (resp. $\prescript{r}{}{\tilde\Psi}_{1,1}^m(M)\subset\tilde\Psi_{1,1}^m(M)$) to be the operators which can be written as ${\rm Op}(a)+S$ with 
$a\in \prescript{r}{}{ S}^m_{1,1}(M)$ (resp. $a\in \prescript{r}{}{\tilde S}^m_{1,1}(M)$) and $S$ a smoothing operator (with $C^\infty$ Schwartz kernel). As in the proof of Lemma \ref{changeofcoord}, we directly see that being in $\prescript{r}{}{\Psi}_{1,1}^m(M)$ or in $\prescript{r}{}{\tilde\Psi}_{1,1}^m(M)$ is independent of the choice of charts 
and coordinates. 
We also denote by  $\Psi^m(M)$ the set of operators that are quantizations of symbols in $S^m(M)$.
The proof of Lemma \ref{changeofcoord} also shows that the principal symbol $\sigma$ is a well-defined linear map
\[ \sigma: \prescript{r}{}{\Psi}_{1,1}^m(M) \to \prescript{r}{}{S}_{1,1}^m(M)/\prescript{r}{}{S}_{1,1}^{m-1}(M)\]
with $\ker \sigma=\prescript{r}{}{\Psi}_{1,1}^{m-1}(M)$.

For~$r>0$, we have a calculus for $\prescript{r}{}{\tilde\Psi}_{1,1}^m(M)$, modulo operators in~$\tilde\Psi_{1,1}^{m-r}(M)$. More precisely, the formulas for the standard quantization in~$\R^n$ will all be truncated at order~$(m-r)$ as explained now:
\begin{prop}{\cite[Theorem 6.2 and Theorem 6.4]{Hormander88}}\label{expansionsymbol}
For $r>0$, the following properties hold:
 \begin{itemize}
  \item Consider~$a\in S^m_{1,1}(\R^n)$ and~$b\in\prescript{r}{}{\tilde S}^{m'}_{1,1}(\R^n)$ with compact support in $x$. Then we have
  \[a(x,D_x)b(x,D_x)=c(x,D_x),\]
  where the symbol~$c\in S^{m+m'}_{1,1}(\R^n)$ satisfies
  \[c(x,\xi)=\sum_{\la\alpha\ra<\ceil{r}}\frac1{\alpha!}\partial^\alpha_\xi a(x,\xi)D^\alpha_xb(x,\xi)+ S^{m+m'-r}_{1,1}(\R^n).\]
  In particular, if $a\in \prescript{r}{}{S}^m_{1,1}(\R^n)$, then $c\in \prescript{r}{}{S}^{m+m'}_{1,1}(\R^n)$.
  \item For $a\in\prescript{r}{}{\tilde S}^m_{1,1}(\R^n)$ with compact support in $x$, we have
  \[a(x,D_x)^*=b(x,D_x),\]
  where~$b\in\prescript{r}{}{\tilde S}^m_{1,1}(\R^n)$ satisfies
  \[b(x,\xi)=\sum_{\la\alpha\ra<\ceil{r}}\frac1{\alpha!}D^\alpha_x\partial^\alpha_\xi \bar a(x,\xi)+ \tilde{S}^{m-r}_{1,1}(\R^n).\]
 \end{itemize}
\end{prop}
\begin{proof} Using the definition of $\prescript{r}{}{\tilde S}^m_{1,1}(\R^n)$, the property for the adjoint follows from \cite[Theorem 6.2]{Hormander88} by taking 
$N=\ceil{r}$ so that $m_N=m-r$ in that Theorem. The property for the composition follows from 
\cite[Theorem 6.4]{Hormander88} by also setting $N=\ceil{r}$. The fact that $c\in \prescript{r}{}{S}^{m+m'}_{1,1}(\R^n)$ uses that $S^{m+m'-r}(\R^n)\subset \prescript{r}{}{S}^{m+m'}_{1,1}(\R^n)$ and the properties \eqref{propertiesclassrPsi}.
\end{proof}
As a direct consequence, on~$M$ we obtain the principal calculus for~$r\geq1$, and the subprincipal one for~$r\geq2$ (recall below that $\sigma(A)$ for $A\in \Psi^m_{1,1}$ 
is always a class modulo $S^{m-1}_{1,1}(M)$).
\begin{prop}\label{composition_mfd}
The following hold true:
 \begin{itemize}
  \item If~$A\in \Psi_{1,1}^m(M)$, resp. $A\in \tilde{\Psi}_{1,1}^m(M)$, and~$B\in\prescript{r}{}{\tilde\Psi}_{1,1}^{m'}(M)$ for $r>0$, then~$AB\in {\Psi}^{m+m'}_{1,1}(M)$, resp. $AB\in {\tilde \Psi}^{m+m'}_{1,1}(M)$, and
  \[\sigma(AB)=\sigma(A)\sigma(B) \, \, {\rm mod }\,  S^{m+m'-r}_{1,1}(M)\]
  If in addition $A\in \prescript{r}{}{\tilde \Psi}_{1,1}^m(M)$, then $AB\in \prescript{r}{}{\tilde \Psi}^{m+m'}_{1,1}(M)$. If moreover $r\geq 1$
  \[\sigma(AB)=\sigma(A)\sigma(B) \, \,{\rm  mod }\, \prescript{r-1}{}{\tilde S}_{1,1}^{m+m'-1}(M).\]
  \item If~$A\in\prescript{r}{}{\tilde\Psi}_{1,1}^m(M)$ and~$B\in\prescript{r}{}{\tilde\Psi}_{1,1}^{m'}(M)$ for $r\in [1,2]$ then~$[A,B]\in\prescript{r-1}{}{\tilde\Psi}_{1,1}^{m+m'-1}(M)$ and
  \[\sigma([A,B])=\frac1i\lB\sigma(A),\sigma(B)\rB \,\, {\rm mod }\,  \tilde{S}^{m+m'-r}_{1,1}(M),\]
  where~$[.,.]$ is the commutator and~$\lB.,.\rB$ the Poisson bracket of functions on~$T^*M$.
    \item If~$A\in\prescript{r}{}{\tilde\Psi}_{1,1}^m(M)$ for $r>0$,
   then for any inner product on~$M$, the adjoint~$A^*\in\prescript{r}{}{\tilde\Psi}_{1,1}^m(M)$ and
   \[\sigma(A^*)=\overline{\sigma(A)} \, \, {\rm mod }\,  \tilde{S}^{m-r}_{1,1}(M).\]
 If $\sigma(A)$ is real and $r\geq 1$, then
 $A^*-A\in \prescript{r-1}{}{\tilde\Psi}_{1,1}^{m-1}(M)$. 
 \end{itemize}
\end{prop}

A direct consequence of this is the following microlocal property:  if $A\in \prescript{r}{}{\tilde\Psi}_{1,1}^m(M)$, $B,B'\in \Psi^0(M)$, then
\begin{equation}\label{prserveWF}
 \WF(B)\cap \WF(B')=\emptyset \Longrightarrow BAB'\in \tilde\Psi_{1,1}^{m-r}(M).
\end{equation} 

The construction of parametrices for elliptic operators being purely symbolic, it still works for our operators.
The \emph{elliptic set} of an operator~$A\in\tilde\Psi_{1,1}^m(M)$ is defined just as in the classical case:
$(x_0,\xi_0)\in {\rm ell}(A)\subset T^*M$ if there is $C>0$ such that $|\xi|^{-m}|\sigma(A)(x,\xi)|\geq C^{-1}>0$ 
for all $(x,\xi)$ in a conic neighborhood $V$ of $(x_0,\xi_0)$ with $|\xi|>C$.
\begin{prop}\label{paramell}
 If~$(x_0,\xi_0)$ is in the elliptic set of~$A\in\prescript{r}{}{\tilde\Psi}_{1,1}^m(M)$, with~$r>0$, there exists~$B\in \Psi_{1,1}^{-m}(M)$, $P\in \Psi^0(M)$ elliptic near $(x_0,\xi_0)$ and $E\in \Psi^{-r}_{1,1}(M)$, with~$BA=P+E$.
\end{prop}
\begin{proof} The proof is contained in \cite[Theorem 3.4.B]{Taylor_nonlinear_book} and follows directly from Proposition \ref{composition_mfd}. 
We sketch the proof for the convenience of the reader. Let $\chi\in C^\infty(T^*M)$ homogeneous of degree $0$ in $|\xi|>1$, supported in a cone where $a$ is elliptic and equal to $1$ near 
$\{(x_0,\lambda\xi_0)\, |\, \lambda>1\}$. 
 Let $b_0(x,\xi):=\chi(x,\xi)/\sigma(A)(x,\xi)$ for $|\xi|$ large enough 
and extend $b_0$ in $T^*M$ in a smooth fashion; note that $b_0\in {^rS}^{-m}_{1,1}(M)$.  We set $B_0={\rm Op}(b_0)$ and $P={\rm Op}(\chi)\in \Psi^0(M)$, then  by Proposition \ref{composition_mfd}  we get $E_0:=B_0A-P\in \Psi^{-1}_{1,1}(M)\cup \Psi^{-r}_{1,1}(M)$. If $r\geq 1$, we set 
$b_1=\chi \sigma(E_0)/\sigma(A)$ for $|\xi|$ large enough and $B_1={\rm Op}(b_1)$ we obtain 
$E_1=B_1A-E_0\in \Psi^{-2}_{1,1}(M)\cup \Psi^{-r}_{1,1}(M)$. We can continue inductively the parametrix and construct $B_j$ for $j\leq r$ and set $B=\sum_{j=0}^{\ceil{r}}B_j$ until we reach a remainder in $\Psi^{-r}_{1,1}(M)$.
\end{proof}
From this and Proposition \ref{boundedness}, one deduces directly the associated elliptic estimates (see \cite[Theorem 3.4.D]{Taylor_nonlinear_book} for reference):
\begin{corr}\label{estimeeell}
Let $A\in\prescript{r}{}{\tilde\Psi}_{1,1}^m(M)$ with $m\geq 0$ and $(x_0,\xi_0)\in {\rm ell}(A)$. Let $u\in H^{s'}(M)$ for some $s'>0$ and assume that $Au\in H^{s}(M)$ for $s>s'-m$. Then for each $Q\in \Psi^0(M)$ microsupported in a small enough conic neighborhood of $(x_0,\xi_0)$, $Qu\in H^{\min(s+m,s'+r)}(M)$ and there is $C>0$ independent of $u$ so that
\[\|Qu\|_{H^{\min(s+m,s'+r)}(M)}\leq C \|Au\|_{H^{s}(M)}+C\|u\|_{H^{s'}(M)}.\]
\end{corr}
\begin{proof}
By Proposition \ref{paramell}, there exists~$B\in \Psi_{1,1}^{-m}(M)$, $P\in \Psi^0(M)$ elliptic near $(x_0,\xi_0)$ and $E\in \Psi^{-r}_{1,1}(M)$ such that~$BA=P+E$. Then 
$Pu=BAu-Eu$ with $BAu\in H^{s+m}(M)$ and $Eu\in H^{s'+r}(M)$ by Proposition \ref{boundedness} (note that $s+m>0$). Using standard pseudo-differential calculus in $\Psi(M)$ and ellipticity of $P$ near $(x_0,\xi_0)$, 
we then obtain the desired result. 
\end{proof}

Next, we state the sharp G\aa rding inequality in that setting: 
\begin{prop}{\cite[Theor\`eme 6.8]{Bony1981},\cite[Theorem 7.1]{Hormander88}}\label{Garding}
Let $A\in\prescript{r}{}{\tilde\Psi}_{1,1}^m(M)$ with~$r\in (0,2]$ and assume that $\Re(\sigma(A))(x,\xi)\geq 0$ for all $x\in M$ and~$\xi$ large enough, then there is $C>0$ such that for any~$u\in C^\infty(M)$,
 \[\Re\ls Au,u\rs\geq-C\lA u\rA^2_{H^{m/2-r/4}(M)}.\]
\end{prop}

\textbf{Extension to operators acting on vector bundles.} As for the usual pseudo-differential operators in the class $\Psi^m(M)$, the theory extends in the obvious way on a smooth vector bundle $E\to M$ equipped with a Hermitian scalar product $\cjg \cdot,\cdot\cjd_E$. The main difference is that the symbols are with values in the endomorphism bundle ${\rm End}(E)=E\otimes E^*\to M$. The only change is the fact that in Proposition \ref{composition_mfd}, one has $[A,B]\in \prescript{r-1}{}{\tilde\Psi}_{1,1}^{m+m'-1}(M)$ only if $[\sigma(A),\sigma(B)]=0$ as elements of ${\rm End}(E)$: this is the case for example if $A$ has principal symbol $a(x,\xi)\otimes {\rm Id}$ for some function $a\in \prescript{r}{}{\tilde S}^{m}_{1,1}(M)$. 
The definition of elliptic set has also to be replaced by $(x_0,\xi_0)\in {\rm ell}(A)$ if and only there is conic neighborhood $V$ of $(x_0,\xi_0)$ such that $\sigma(A)(x,\xi)$ is invertible in ${\rm End}(E)$ for $|\xi|$ large enough. 
For the sharp G\aa rding inequality (Proposition \ref{Garding}), the condition $\Re(\sigma(A))(x,\xi)\geq 0$ has to be understood as $\sigma(A)(x,\xi)+\sigma(A)^*(x,\xi)\geq 0$ in the sense of symmetric endomorphisms on $E$ for the scalar product $\cjg\cdot,\cdot\cjd_E$; indeed, the proof of \cite{Hormander88}, strongly based on the proof of the classical case \cite[Theorem 18.1.14]{HormanderVol3} applies equally for quantization of symbols with values in Banach spaces, as mentioned in \cite[Remark 2, Page 79]{HormanderVol3}.

\subsection{Paradifferential Calculus}
The paradifferential calculus is a way of regularising operators with rough coefficients, which turns out to be well suited to study non linear expressions. Since we shall only need this case, we will consider the case of differential operators.

Following \cite[Section 1.3]{Taylor_nonlinear_book}, let us first introduce the class of symbols with rough coefficients: we say that $a\in C^{r}S^{m}_{1,\delta}(M)$ with $r\geq 0$ and $\delta\in [0,1]$ if $\pl_\xi^\beta a\in C^r(T^*M)$ for all $\beta$ with the following bounds in local charts
\[\forall \beta,\exists C_\beta>0,\quad \|\pl_\xi^\beta a(\cdot,\xi)\|_{C^r}\leq C_\beta (1+|\xi|)^{m-|\beta|+r|\delta|}\]
with the usual convention that when $r\in \N$, the $C^r$ norm involves the $C^0$ norm of the $j\leq r$ derivatives. We shall also denote $C^rS^m(M):=C^rS^m_{1,0}(M)$.

\begin{defi}
For an order~$m\in \N$, and an index~$r\geq m$, $\prescript{r}{}{\rm Diff}^m(M)$ denotes the set of differential operators $P$ of order $m$ on $M$ which in local charts can be written under the form
\[ P=\sum_{|\alpha|\leq m}p_{\alpha}(x)\pl_{x}^\alpha\]
with $p_\alpha\in C^{r-(m-|\alpha|)}(M)$. They can also be written under the form ${\rm Op}(p)$
where  $p\in C^{r-m}(T^*M)$ is polynomial of order $m$ in the fibers, and given in the charts by 
$p(x,\xi)=\sum_{|\alpha|\leq m}p_\alpha(x)(i\xi)^\alpha$. Note that $p$ belongs to $\sum_{k=0}^m C^{r-k}S^{m-k}(M)$.
\end{defi}
This class contains the set of differential operators of order $m$ with $C^r$ coefficients.
We also remark that if $P\in \prescript{r}{}{\rm Diff}^m(M)$, then $P^*\in \prescript{r}{}{\rm Diff}^m(M)$ so that this class is stable by taking adjoint, contrary to the space of differential operators with $C^r$ coefficients.

\begin{lemm}\label{regularization}
 There is a continuous linear map sending each differential operator $P={\rm Op}(p)\in \prescript{r}{}{\rm Diff}^m(M)$ to a symbol 
$p^\sharp\in \prescript{r}{}{\tilde S}^m_{1,1}(M)$, in such a way that 
$p^\flat:=p-p^\sharp \in C^{r-m}S_{1,1}^{m-r}(M)$ and $p^\flat\in \sum_{k=0}^m\cap_{s\in (0,r-k)}C^{r-k-s}S^{m-k-s}(M)$. Moreover, if $m\in \{0,1\}$,
\begin{equation}\label{boundroughcoef}
\forall s\in (0,r), \quad {\rm Op}(p^\flat): H^{s+m-r}(M)\to H^s(M)
\end{equation}
is bounded.
\end{lemm}
\begin{proof}
The proof is done in \cite[Section 2]{Bony1981}, \cite[Section 3.2]{Taylor_nonlinear_book} or \cite[Chapter X]{Hormander-nonlinear} in $\R^n$, we just have to proceed similarly using local charts, as is done in \cite[Section 4.1]{Cekic-Guillarmou}\footnote{In this paper, the regularization is with respect to a semi-classical parameter rather than $|\xi|$}. 
We fix a covering with small charts $U_j\simeq B(0,\eps)\subset \R^n$ 
and an associated partition of unity $(\psi_j)_j$. In each chart $U_j$, a function $a_j\in C_c^r(U_j)$
can be regularized as in \cite[Section 3.2]{Taylor_nonlinear_book}: let $\chi \in C^\infty(\R^n\times\R^n)$ so that $\chi(\eta,\xi)=0$ for $|\eta|>|\xi|/2$ and $\chi(\eta,\xi)=1$ if 
$|\eta|<|\xi|/16$, and set for $x\in U_j, \xi\in \R^n$  
\[ a_j^\sharp(x,\xi)=\mc{F}_{\eta\to x}^{-1}(\hat{a}_j(\eta)\chi(\eta,\xi)).\]
Then we define $a^\sharp=\sum_j\psi_ja_j^\sharp$ where $a_j:=a|_{U_j}\times \tilde{\psi}_j$ and $\tilde{\psi}_j\in C_c^\infty(U_j)$ is equal to $1$ on ${\rm supp}(\psi_j)$. For a differential operator $p$, we proceed similarly by writting $P={\rm Op}(p)$ and $p(x,\xi)|_{U_j} \tilde{\psi}_j(x)=\sum_{|\alpha|\leq m}a_{j,\alpha}(x)(i\xi)^\alpha$ and setting
\begin{equation}\label{formulapsharp} 
p^\sharp(x,\xi)=\sum_{j}\sum_{|\alpha|\leq m}\psi_j(x)a^\sharp_{j,\alpha}(x,\xi)(i\xi)^\alpha.
\end{equation}
According to \cite[Proposition 3.2.1, Proposition 1.3.B]{Taylor_nonlinear_book}, one has, setting $k=m-|\alpha|$
\[a_{j,\alpha}^\sharp\in \prescript{r-k}{}{\tilde S}_{1,1}^0(U_j), \quad \psi_j (a_{j,\alpha}-a_{j,\alpha}^\sharp) \in C^{r-k}S_{1,1}^{-r+k}(U_j)\cap C^{r-k-s}S^{-s}(U_j) \]
for $s\in (0,r-k)$. 
Thus $p^\sharp \in \prescript{r}{}{\tilde S}_{1,1}^m(M)$ and $p^\flat=p-p^\sharp$ has the announced regularity. For \eqref{boundroughcoef}, we refer to \cite[Theorem 2.1.A]{Taylor_nonlinear_book}.
\end{proof}
\textbf{Remark.} The regularization procedure described above is of course depending on the choice of cutoff $\chi$ but two regularizations of a symbol $a\in C^rS^m(M)$ yield the same symbol $a^\sharp\in \prescript{r}{}{\tilde S}_{1,1}^m(M)$ modulo the class $\tilde{S}_{1,1}^{m-r}(M)$: this is proved in \cite[Theor\`eme 2.1]{Bony1981} or in \cite[Proposition 10.2.2.]{Hormander-nonlinear}.
Another equivalent way of regularizing can be done using Littlewood-Paley decomposition with the same exact properties, see again \cite[Theorem 2.1]{Bony1981} or \cite[Section 3.2]{Taylor_nonlinear_book}. We can thus use freely both constructions for the paradifferential operator, keeping in mind that some proof are sometime more transparent with one definition than the other.

\begin{defi}
For a differential operator $P={\rm Op}(p)\in \prescript{r}{}{\rm Diff}^m(M)$, we define its associated paradifferential operator to be 
\[ T_p:= {\rm Op}(p^\sharp)\in \prescript{r}{}{\tilde \Psi}^m_{1,1}(M).\]
\end{defi}
Although we shall not use it, one can more generally define a paradifferential operator associated to each symbol $a\in C^rS^m(M)$ of order $m$, see \cite[Chapter X]{Hormander-nonlinear} or \cite{Taylor_nonlinear_book}.

By definition, for $P={\rm Op}(p)\in \prescript{r}{}{\rm Diff}^m(M)$,
\begin{equation}\label{symbolpara}
\sigma(T_p)(x,\xi)=p^\sharp(x,\xi) \,\, {\rm mod}\,\,  \prescript{r}{}{\tilde S}_{1,1}^{m-1}(M)
\end{equation}
and, if $r\geq 1$, viewing this principal symbol as an element in $C^{r}S_{1,1}^m(M)/C^{r-1}S_{1,1}^{m-1}(M)$, we recover by Lemma \ref{regularization} the principal symbol of $P$ (which has rough regularity)
\[ \sigma(T_p)=\sigma(P) \textrm{ mod } C^{r-1}S_{1,1}^{m-1}(M).\] 
This is particularly important for the propagation estimates as the Hamilton flow of $\sigma(T_p)$ 
for $|\xi|$ very large becomes asymptotically the Hamilton flow of $\sigma(P)$.

We then get a calculus, which is traditionally written as follows.
\begin{prop}{\cite[Theor\`emes 3.2 and 3.3]{Bony1981}, \cite[Theorem 10.2.4 and 10.2.5]{Hormander-nonlinear}}\label{composePara}
 Let~$P={\rm Op}(p)\in \prescript{r}{}{\rm Diff}^m(M)$ and~$Q={\rm Op}(q)\in  \prescript{r}{}{\rm Diff}^l(M)$ with $r>\max(m,l)$. Then
 \begin{itemize}
  \item $T_{pq}-T_pT_q\in \prescript{r-1}{}{\tilde \Psi}^{m+l-1}_{1,1}(M)$ if~$r\geq 1$;
  %\item $T_pT_q-T_qT_p-\frac1iT_{\lB a,b\rB}\in \prescript{r-2}{}{\tilde \Psi}_{1,1}^{m+l-2}$ if~$r\geq 2$;
  \item $T_p^*-T_{\bar p}\in \prescript{r-1}{}{\tilde \Psi}_{1,1}^{m-1}(M)$ if~$r\geq 1$.
 \end{itemize}
\end{prop}
\begin{proof}
Using Proposition \ref{expansionsymbol}, the proof reduces to showing that $\sigma(T_{pq})=\sigma(T_p)\sigma(T_q)$ and $\sigma(T_p^*)=\sigma(T_{\bar{p}})$. 
Since $\sigma(T_pT_q)=
\sigma(T_p)\sigma(T_q)=p^\sharp q^\sharp \in 
\prescript{r}{}{\tilde S}_{1,1}^{m+l}(M)/\prescript{r-1}{}{\tilde S}_{1,1}^{m+l-1}(M)$ by Proposition \ref{composition_mfd}, this reduces to showing that $(pq)^\sharp=p^\sharp q^\sharp$ modulo $\prescript{r-1}{}{\tilde S}_{1,1}^{m+l-1}(M)$. In view of our regularization definition, this fact reduces in local charts to the case in $\R^n$, which is proved in \cite[Theorem 10.2.5]{Hormander-nonlinear}. The second statement is similar.
\end{proof}

One of the main properties of the paradifferential operators in our setting is the decomposition of products of non-smooth functions, called \emph{paraproducts}. The idea, for $a,b$ in H\"older or Sobolev classes, is to replace $ab$ by some paradifferential operators up to smoother terms. 

\begin{prop}{\cite[Theorem 2.5]{Bony1981}, \cite[Section 3.5]{Taylor_nonlinear_book}, \cite[Theorem 10.2.8]{Hormander-nonlinear}}\label{paraproduct}.
 If~$a$ and~$b$ are~$L^\infty$ functions on~$M$, then
 \[ab=T_ab+T_ba+R(a,b),\]
 where the bilinear symmetric operator~$R$ as the following mapping properties:
 \begin{itemize}
  \item $R:C^r(M)\times H^s(M)\rightarrow H^{s+r}(M)$ if~$s>0,r>0$,
  \item $R:H^s(M)\times H^t(M)\rightarrow H^{s+t-\frac n2}(M)$ if~$s+t>n/2$,
  \item $R:C^r(M)\times C^\rho(M)\rightarrow C^{r+\rho}(M)$ if~$r,\rho>0$ and~$r+\rho$ is not an integer.
 \end{itemize}
In addition, if $a\in C^r(M)$, then $R_a:=R(a,\cdot)\in \Psi_{1,1}^{-r}(M)$.
\end{prop}
\begin{proof} First we recall the results  of \cite[Theorem 2.5]{Bony1981}, \cite[Section 3.5]{Taylor_nonlinear_book} in $\R^n$: $a_jb_j={\rm Op}_{\R^n}(a_j^\sharp)b_j+{\rm Op}_{\R^n}(b_j^\sharp)a_j+R_{\R^n}(a_j,b_j)$ for any $a_j,b_j$ with compact support in $\R^n$ and $R_{\R^n}$ is bounded as claimed in the Proposition. 
It then suffices to write $ab=\sum_{j}\psi_j (\tilde{\psi}_ja)(\tilde{\psi}_jb)$ with $\psi_j,\tilde{\psi}_j$ as in the proof of Lemma \ref{regularization}, 
and we deduce, with $a_j=\tilde{\psi}_ja$ and $b_j=\tilde{\psi}_jb$, that 
\[\begin{split}
ab=& \sum_j \psi_ja_jb_j=\sum_{j} {\rm Op}_{\R^n}(a_j^\sharp)\tilde{\psi}_jb+{\rm Op}_{\R^n}(b_j^\sharp)\tilde{\psi}_j a +\psi_jR_{\R^n}(a_j,b_j)\\
=&  T_ab+T_ba+\sum_j \psi_jR_{\R^n}(a_j,b_j).
\end{split}\]
The fact that $R_a\in \Psi_{1,1}^{-r}(M)$ is proved in \cite[Theorem 10.2.8]{Hormander-nonlinear}.
\end{proof}

For a distribution $u\in C^{-\infty}(M)$, 
denote by ${\rm WF}_{H^s}(u)\subset T^*M\setminus\{0\}$ (where $\{0\}$ denotes the zero section) the complement in $T^*M\setminus\{0\}$ of those $(x_0,\xi_0)$ such that there is a conic (in $\xi$) neighborhood  $U\subset T^*M\setminus\{0\}$ of $(x_0,\xi_0)$ such that for all $A\in \Psi^0(M)$
with ${\rm WF}(A)\subset U$, one has $Au\in H^s(M)$. 
Note that $\WF(u)=\WF_{C^\infty}(u)$ and that, as usual with wavefront sets, using a partition of unity, cutoffs functions with small supports in charts, one can reduce the analysis to $\rr^n$ where $A=a(D)$ are Fourier multipliers.

Recall the well-known formula $\mathrm{WF}(ab)\subset (\mathrm{WF}(a)+\mathrm{WF}(b))$. 
The paraproduct decomposition induces a similar decomposition of the wavefront set, so that the wavefront set of $R(a,b)$ is included in $(\mathrm{WF}(a)+\mathrm{WF}(b))$. More specifically, we have the following result.
\begin{prop}\label{WF(R)}
Let $\eps\in (0,1), \alpha>0, \beta>0, \delta>0$,  and assume that $a,b\in C^\eps(M)$ and $a\in H^{\alpha}(M)$ and $b\in H^\beta(M)$. Then one has 
\[\mathrm{WF}_{H^{\min(\alpha,\beta)+\delta+\eps}}(R(a,b))\subset\mathrm{WF}_{H^{\alpha+\delta}}(a)+\mathrm{WF}_{H^{\beta+\delta}}(b).\]
\end{prop} 
\begin{proof}
	The result on $M$ can be classically deduced from the corresponding result in $\R^n$ by using charts and a partition of unity. Then by taking the supports small enough we can assume that we have closed cones $K_a,K_b\subset \R^n\setminus\{0\}$ in frequency space such that $\mathrm{WF}_{H^{\alpha+\delta}}(a) \subset \R^n\times K_a$ and $\mathrm{WF}_{H^{\beta+\delta}}(b)\subset \R^n\times K_b$ and prove that $R(a, b)\in H^{\min(\alpha,\beta)+\delta+\eps}(\R^n)$ away from $\R^n\times (K_a+K_b)$. We shall use the definition of $T_a$ using the Littlewood-Paley decomposition 
	
With those reductions in mind, we recall that we can use Littlewood-Paley decomposition to write $a = \sum_{j=-1}^\infty a_j$, with $\hat{a}_{-1}$ supported near $0$, and each other $\hat{a}_j$ supported in the dyadic annulus $\mathcal{C}_j=2^j\mathcal{C}$ where the annulus $\mathcal{C}$ around the unit sphere is such that each $\mathcal{C}_j$ intersects only $\mathcal{C}_{j-1}$ and $\mathcal{C}_{j+1}$ (hat denotes Fourier transform). Let $\rho_a\in C^\infty(\rr^n)$ (resp. $\rho_b\in C^\infty(\rr^n)$) be a smooth function, homogeneous of degree $0$ for large $\xi$, equal to $1$ near $K_a\cap \{\xi\geq 1\}$ (resp. $K_b\cap \{|\xi|\geq 1\}$) and $0$ away from a conic neighborhood of $K_a$ (resp. $K_b$). 
Since $\WF_{H^{\alpha+\delta}}(a)\subset \rr^n\times K_a$,	the sequence $2^{j (\alpha+\delta)}\lA (1-\rho_a(D))a_j\rA_{L^2}$ is in $\ell^2(\nn)$ (here $\rho_a(D)$ means the Fourier multiplier by $\rho_a(\xi)$).  
Similarly, writting $b=\sum_{j=0}^\infty b_j$ for the Littlewood-Paley decomposition of $b$, $\WF(b)\subset \rr^n\times K_b$ implies that the sequence $2^{j (\beta+\delta)}\lA (1-\rho_b(D))b_j\rA_{L^2}$ is in $\ell^2(\nn)$. The statement $a\in C^\eps$ is equivalent to saying that the sequence $2^{j \eps}\lA a_j\rA_{L^\infty}$ is in $\ell^\infty(\nn)$, with equivalence of the norms.
	
Now, using the Littlewood-Paley definition $T_au:=\sum_{k=-1}^\infty \sum_{j=-1}^{k-1}a_ju_k$ for the paradifferential operator, one can write $R = R(a, b) = \sum_{|j-k|<2}a_jb_k$. We see immediately from support considerations that the $\ell$-th Littlewood-Paley block of $R$, written $R_\ell$, is a sum of $a_jb_k$ with $|j-k|<2$ and $j>\ell-N_0$ for some fixed $N_0$.
	We choose a cut-off function $\chi_R$, homogeneous of degree $0$ for $|\xi|\geq 1$, supported away from $(K_a+K_b)$, and we look at the $L^2$ norm of $\chi_R(D)a_jb_k$ for such a piece $a_jb_k$, which is
\[\begin{split}
\chi_R(a_jb_k)=& \chi_R(D)((\rho_a(D)a_j)(\rho_bb_k)) + \chi_R(D)(((1-\rho_a(D))a_j)(\rho_b(D)b_k)) \\
		  & +\chi_R(D)((\rho_a(D)a_j)((1-\rho_b(D))b_k)) + \chi_R(D)(((1-\rho_a(D))a_j)((1-\rho_b(D))b_k)).
\end{split}\]
	The first term is zero since its Fourier transform is $\chi_R(\widehat{\rho_aa_j}\star\widehat{\rho_bb_k})$ and the support of the convolution is in $K_a+K_b$.
	
For the second term, there is some uniform constant $C>0$ such that
\begin{align*}
& \lA\chi_R(D)(((1-\rho_a(D))a_j)(\rho_b(D)b_k))\rA_{L^2}\leq\lA((1-\rho_a(D))a_j)\rA_{L^2}\lA b_k\rA_{L^\infty}\\
		&\leq C2^{-j(\alpha+\delta+\eps)}\left(2^{j(\alpha+\delta)} \lA((1-\rho_a(D))a_j)\rA_{L^2}\times2^{j\eps}\lA b_k\rA_{L^\infty}\right) \\
		&\leq C 2^{-j(\alpha+\delta+\eps)}\left(2^{j(\alpha+\delta)} \lA((1-\rho_a(D))a_j)\rA_{L^2}\times2^{k\eps}\lA b_k\rA_{L^\infty} \right)
	\end{align*}	
	where we have used  that multipliers of degree 0 are bounded on $L^2$, $\lA\rho_b(D)b_k\rA_{L^\infty}\leq C\lA b_k\rA_{L^\infty}$ with a constant $C$ independent of $k$ (this holds since the Fourier support of $b_k$ is in an annulus and $\rho_b$ is homogeneous of degree $0$ for large $\xi$) and, at the last line, that $k$ and $j$ are comparable. 
	
	For the third term, we have similarly
\begin{align*}
		&\lA\chi_R(D)((\rho_a(D)a_j)((1-\rho_b(D))b_k))\rA_{L^2}\leq C\lA a_j\rA_{L^\infty}\lA(1-\rho_b(D))b_k\rA_{L^2}\\
		&\leq C2^{j(\beta+\delta+\eps)}\left(2^{j\eps} \lA a_j\rA_{L^\infty}\times 2^{j(\beta+\eps)}\lA(1-\rho_b(D))b_k\rA_{L^2}\right) \\
		&\leq C 2^{-j(\beta+\delta+\eps)}\left(2^{j\eps} \lA a_j\rA_{L^\infty}\times2^{k(\beta+\delta)}\lA(1-\rho_b(D))b_k\rA_{L^2}\right).
	\end{align*}
The bounds on the last term are proved in the same way. In the end, we find
\begin{equation*}
	2^{\ell(\min(\alpha,\beta)+\eps)}\lA\chi_R(D)R_\ell\rA_{L^2}\leq C\sum_{j>\ell-N_0, |\nu|\leq 1}2^{-(j-\ell)(\min(\alpha,\beta)+\delta+\eps)}A_jB_{j-\nu}
\end{equation*}
where $A_j$ is an $\ell^\infty$ sequence and $B_k$ is an $\ell^2$ sequence. Writing $\gamma:=\min(\alpha,\beta)+\delta+\eps$, we have by Cauchy-Schwarz inequality 
\[\sum_{\ell=0}^\infty 2^{\ell(\min(\alpha,\beta)+\eps)}\lA\chi_R(D)R_\ell\rA_{L^2}^2\leq C\sum_{\ell=0}^{\infty} \sum_{k=-N_0,|\nu|\leq 1}^\infty 2^{-k\gamma}B_{\ell+k-\nu}^2\leq  C \sum_{k\geq 0} B_k^2<\infty
\] 
The result is proved.
\end{proof}

\subsection{Propagation estimates}
In this section, we will show that the usual propagation of singularity estimates for classical pseudo-differential operators $P\in \Psi^m(M)$ apply to paradifferential operators, as well as the radial type estimates (source and sink). The proofs are essentially the same for paradifferential operators once we have the material of the previous section. For convenience of the reader we shall give some details. 

First, we come back to the notion of principal symbol for $T_p$ if $P\in \prescript{r}{}{\rm Diff}^m(M)$.
It is convenient to consider the  fiber radial-compactification $\bbar{T}^*M$ 
of the cotangent bundle, see \cite[Section E.1.3]{DZ19} for details. It amounts to adding the sphere bundle $S^*M=T^*M\setminus\{0\}/\rr^+$ at infinity in $\xi$, and it becomes a ball-bundle with boundary $\pl\bbar{T}^*M=S^*M$. If the principal symbol $\sigma(A)$ of an operator $A\in \Psi^m(M)$ is homogeneous of degree $m$, it can then be viewed as a function $(|\xi|^{-m}\sigma(A))|_{\pl \bbar{T}^*M}$. Similarly, if $P\in  \prescript{r}{}{\rm Diff}^m(M)$, one recovers the principal symbol of $T_p$ modulo $C^{r-m}S^{m-\min(1,r)}_{1,1}(M)$:
\[ \sigma(T_p)= (|\xi|^{-m}\sigma(T_p))|_{\pl \bbar{T}^*M}=(|\xi|^{-m}\sigma(P))|_{\pl \bbar{T}^*M}.\]
Moreover, we notice that if $r>1$, $m=1$ and the principal symbol $p_1(x,\xi):=\sigma(P)(x,\xi)$
is real-valued, then $P$ can be written as $P=-iX+V$ where $X$ is a $C^r$-vector field defined by $p_1(x,\xi)=\xi(X(x))$ and where $V\in C^{r-1}(M)$ a potential. Moreover the flow $\varphi_t$ of $X$ is well-defined and with regularity $C^{r}(M,M)$. As explained in \cite[1st paragraph of Section 4.2]{Cekic-Guillarmou}, the Hamilton 
vector field $H_{p_1}$ of $p_1(x,\xi)$ and its flow $e^{tH_{p_1}}$ are also well-defined and given by the symplectic lift 
\[\Phi_t(x,\xi)=e^{tH_{p_1}}(x,\xi)=(\varphi_t(x),(d\varphi_t(x)^{-1})^T\xi).\]   
Since $p_1$ is homogeneous of degree $1$, then $H_{p_1}$ is homogeneous of degree $0$ and it extends, as well as its flow, as a $C^{r-1}$ vector field and flow on $\bbar{T}^*M$. Note also that $H_{p_1}f=Xf$ if $f$ is independent of $\xi$ (i.e. the pull-back of a function on $M$).

We start with a technical Lemma.
\begin{lemm}\label{approximationlemma}
Let $H_{p_1}$ be the Hamilton vector field of $p_1(x,\xi)=\xi(X(x))$ on $T^*M$ with $X\in C^r(M;TM)$ for $r>1$. Let $u\in C^{r-1}(T^*M)$, smooth in the $\xi$ variable in the sense that $\pl_{\xi}^\beta u\in C^{r-1}(T^*M)$ for all $\beta$, and homogeneous of degree $0$ in $\xi$ for $|\xi|$ large. If $H_{p_1}u\in C^{r-1}(T^*M)$, then  for all $\eps>0$ there exists $v\in S^0(M)$, homogeneous of degree $0$ in $\xi$ for  $|\xi|$ large, so that $\|u-v\|_{L^\infty}<\eps$ 
and $\|H_{p_1}(u-v)\|_{L^\infty}<\eps$. Moreover, if $u\geq 0$, we can choose $v\geq 0$.
\end{lemm}
\begin{proof} First, using a partition of unity, remark that it suffices to assume that $u$ is supported in $T^*U_j$ where $U_j\subset M$ is a chart.
Then, observe that in local coordinates in the chart
$H_{p_1}=X(x)-\sum_{k}\pl_{x_k}(\xi(X(x)))\pl_{\xi_k}$ if $p_1(x,\xi)=\xi(X(x))$ for $X\in C^{r}(M;TM)$. Since $u$ and $H_{p_1}u$ are homogeneous of degree $0$ for $|\xi|$ large, 
it suffices to consider the bound for $|\xi|$ bounded, provided we take $v$ homogeneous of degree $0$ for $|\xi|$ large. We set, after identifying $U_j\simeq B(0,1)\subset \R^n$, for $\chi\in C_c^\infty(B(0,1);[0,1])$ with $\int \chi=1$  
\[ R_\eps u(x,\xi):=\eps^{-n}\int_{\R^n}\chi(\frac{x-y}{\eps})u(y,\xi)dy.\]
It is a routine exercise to see that $\|R_\eps u-u\|_{L^\infty}=\mc{O}(\eps^s\|u\|_{C^s})$ if $u\in C_c^s(B(0,1)\times K)$ for $s\in (0,1)$, where $K\subset \R^n$ is a compact set (in the $\xi$ variable). Moreover $R_\eps u\geq 0$ if $u\geq 0$ and $R_\eps$ commutes with all $\pl_{\xi_k}$.
We have $X(R_\eps-1) =(R_\eps-1) X+B_\eps$ where $B_\eps:=[X, R_\eps]$ and 
$\|(R_\eps-1) Xu\|_{L^\infty}=\mc{O}(\eps^{r-1}\|Xu\|_{C^{r-1}})$ by using that $Xu\in C^{r-1}$. 
Let us then study $B_\eps$. Writing $X(x)=\sum_{k}X_k(x)\pl_{x_k}$ and using that $X_k\in C^{r}(B(0,1))$, we get for each $f\in C^\infty$
\[\begin{split}
B_\eps f(x,\xi)=& \sum_k\eps^{-n}\int_{\R^n}\Big(\chi(\frac{x-y}{\eps})\pl_{y_k}X_k(y)+(\pl_{y_k}\chi)(\frac{x-y}{\eps})\frac{X_k(x)-X_k(y)}{\eps}\Big)f(y,\xi)dy\\
=& \sum_k\pl_{x_k}X_k(x)f(x,\xi)+\sum_j \pl_{x_j}X_k(x)\int \pl_{z_k}\chi(z)z_j f(x-\eps z,\xi)dz+\mc{O}(\eps^{r-1}\|f\|_{C^{r-1}})\\
=& \mc{O}(\eps^{r-1}\|f\|_{C^{r-1}}),
\end{split}\]
where we used $R_\eps f=f+\mc{O}(\eps^{r-1}\|f\|_{C^{r-1}})$ and 
$(\pl_{x_k}X_k)(y)=(\pl_{x_k}X_k)(x)+\mc{O}(\eps^{r-1})$ for $|x-y|<\eps$ in the second line, and 
$f(x-\eps z,\xi)=f(x)+\mc{O}(\eps^{r-1}\|f\|_{C^{r-1}})$ together with integration by parts in the third line. We conclude that $\|H_{p_1}(u-R_\eps u)\|_{L^\infty}=\mc{O}(\eps^{r-1})$ and we obtain the desired result by setting $v:=R_{\delta(\eps)}u$ for $\delta(\eps)=\eps^{\frac{1}{r-1}}$ with $\eps>0$ small.
\end{proof}

Bony proved in \cite[Theorem 6.2]{Bony1981} a propagation of singularity result for real principal type operators $T_p\in \prescript{r}{}{\tilde \Psi}_{1,1}^{m}(M)$ with $r>1$ and $P={\rm Op}(p)\in \prescript{r}{}{\rm Diff}^m(M)$. We recall the proof for the reader's convenience for the case $m=1$ (which is our case of interest). 
\begin{prop}{\cite[Theorem 6.2 and 6.2']{Bony1981}}\label{propofsing}
Let~$P^\sharp:=T_p\in \prescript{r}{}{\tilde \Psi}_{1,1}^1(M)$, with~$r\in (1,2)$ and $P={\rm Op}(p)\in \prescript{r}{}{\rm Diff}^1(M)$ and assume that $p_1:=\sigma(P)$ is real-valued. 
Let
\[\Phi_t=\exp(tH_{p_1}): T^*M\rightarrow T^*M\]
be the flow of the Hamilton vector field of $p$.
Let~$A$, $B$, $B_1$ in~$\Psi^0(M)$, such that 
for~$(x,\xi)\in\mathrm{WF}(A)$, there is some~$T\geq0$ such that
\begin{equation}\label{controlcond}
\Phi_{-T}(x,\xi)\in\mathrm{ell}(B);\quad\Phi_t(x,\xi)\in\mathrm{ell}(B_1)\text{ for all }t\in[-T,0].\end{equation}
Then for all~$s\in\R$, all~$\eps>0$, there is $C>0$ such that for all $u\in H^{s-(r-1)+\eps}(M)$ satisfying
$Bu\in H^s(M)$ and $B_1f\in H^s(M)$ with $P^\sharp u=f$, we have $Au\in H^s(M)$ and 
\begin{equation}\label{estimeeProp}
\lA Au\rA_{H^s(M)}\leq C\lA Bu\rA_{H^s(M)}+C\lA B_1f\rA_{H^s(M)}+C\lA u\rA_{H^{s-(r-1)+\eps}(M)}.
\end{equation}
\end{prop}
\begin{proof} We follow the proof in \cite[Theorem E.47]{DZ19} given for operators in $\Psi^m(M)$ and will focus mainly on the differences due to our assumptions. Without loss of generality, we may assume that $B_1$ has non-negative principal symbol, is self-adjoint and ${\rm WF}(1-B_1)\cap \cup_{t=0}^{T}\Phi_{-t}({\rm WF}(A))=\emptyset$.
As explained above, the Hamilton field $H_{p_1}$ and its flow $\Phi_t$ extend to $\bbar{T}^*M$ in a  $C^{r-1}$ fashion.
By \cite[Lemme 6.6]{Bony1981}, for each $\beta\geq 0$, there is $g\in C^\infty(\bbar{T}^*M, \rr^+)$ homogeneous of degree $0$ such that ${\rm supp}(g)\subset {\rm ell}(B_1)$ with $g>0$ on ${\rm WF}(A)$ and $H_pg\leq -\beta g$ on $\bbar{T}^*M\setminus {\rm ell}(B)$. Set $G={\rm Op}(\cjg \xi\cjd^{s}g)\in \Psi^{s}(M)$ with ${\rm WF}(G)\subset {\rm ell}(B_1)$. Take $u\in C^\infty(M)$, $f=P^\sharp u$ and 
write
\[ {\rm Im}\cjg f,G^*Gu\cjd={\rm Im}\cjg {\rm Re}(P^\sharp)u,G^*Gu\cjd+{\rm Re}\cjg {\rm Im}(P^\sharp)u,G^*Gu\cjd\] 
The first term in the RHS can be written as $\cjg Zu,u\cjd$ where $Z:=\frac{i}{2}[{\rm Re}(P^\sharp),G^*G]$. Since ${\rm Re}(P^\sharp) \in \prescript{r}{}{\tilde \Psi}_{1,1}^1(M)$ with $r>1$, Proposition \ref{composition_mfd} insures that $Z\in \prescript{r-1}{}{\tilde \Psi}_{1,1}^{2s}(M)$. Moreover, using Proposition \ref{composition_mfd} and \eqref{symbolpara}, $Z$ has  principal symbol 
\[ \sigma(Z)=\cjg \xi\cjd^{2s}( gH_{p_1}g+s H_{p_1}(\log \cjg \xi\cjd) g^2)+o(|\xi|^{2s})\]
as $|\xi|\to \infty$.
Thus there is $C_1>0$ such that $\sigma(Z+(\beta-C_1)G^*G)\leq 0$ near $\bbar{T}^*M\setminus {\rm ell}(B)$ for $|\xi|$ large enough, and by the sharp G\aa rding inequality (Proposition \ref{Garding}) applied to $B_1ZB_1$, then for each $N>0$ there is $C>0$ so that
\[ \cjg ZB_1u,B_1u\cjd\leq (C_1-\beta)\|Gu\|_{L^2(M)}^2+C\|Bu\|^2_{H^s(M)}+C\|B_1u\|^2_{H^{s-\frac{r-1}{4}}(M)}+C\|u\|_{H^{-N}(M)}.\]
Note that $(1-B_1)Z(1-B_1)\in \Psi^{-\infty}(M)$, 
$(1-B_1)ZB_1\in \tilde{\Psi}_{1,1}^{2s-r-1}(M)$ and $B_1Z(1-B_1)\in \tilde{\Psi}_{1,1}^{2s-(r-1)}(M)$ using \eqref{prserveWF}, thus for each $\eps_0\in (0,(r-1)/2)$
\[ |\cjg (Z-B_1ZB_1)u,u\cjd|\leq C\|B_1u\|^2_{H^{s-\eps_0}(M)}+C\|u\|^2_{H^{s-(r-1)+\eps_0}(M)}.\]
Next, we deal with $\cjg {\rm Im}(P)u,G^*Gu\cjd$:  
first we have 
\[ |\cjg {\rm Im}(P^\sharp)u,G^*Gu\cjd| \leq |\cjg {\rm Im}(P^\sharp)Gu,Gu\cjd|+\cjg G^*[G,{\rm Im}(P^\sharp )]u,u\cjd|.\]
Using that $p_1=\sigma(P^\sharp)$ is real, we have ${\rm Im}(P^\sharp)\in \prescript{r-1}{}{\tilde \Psi}^{0}_{1,1}(M)$ by Proposition \ref{composePara}, and using Proposition \ref{composition_mfd} we get 
$G^*[G,{\rm Im}(P^\sharp)]\in  {\tilde \Psi}^{2s-(r-1)}_{1,1}(M)$. Together with Proposition \ref{boundedness} and the fact that ${\rm WF}(G)\subset {\rm ell}(B_1)$, we obtain for $\eps_0>0$ as above
\[|\cjg {\rm Im}(P^\sharp)u,G^*Gu\cjd| \leq C_2\|Gu\|^2_{L^2(M)}+C\|B_1u\|^2_{H^{s-\eps_0}(M)}+C\|u\|^2_{H^{s-(r-1)+\eps_0}(M)}\]
for some constant $C_2>0,C>0$. Therefore 
\[|{\rm Im}\cjg f,G^*Gu\cjd|\leq (C_1+C_2-\beta)\|Gu\|_{L^2(M)}^2+C\|Bu\|^2_{H^s(M)}+C\|B_1u\|^2_{H^{s-\eps_0}(M)}+C\|u\|^2_{H^{s-(r-1)+\eps_0}(M)}. \]
Taking $\beta=C_1+C_2+1$, and using $|{\rm Im}\cjg f,G^*Gu\cjd|\leq C\|B_1f\|_{H^s}^2+\|Gu\|^2_{L^2}$ we deduce that
\begin{equation}\label{boundGu} 
\|Gu\|_{L^2}^2\leq C\|B_1f\|^2_{H^s}+
C\|Bu\|^2_{H^s}+C\|B_1u\|^2_{H^{s-\eps_0}}+C\|u\|^2_{H^{s-(r-1)+\eps_0}}.
\end{equation}
Since $\|Au\|_{H^{s}(M)}\leq C\|Gu\|_{L^2}$ for some $C>0$ by elliptic estimates, we get
\[\|Au\|_{L^2}^2\leq C\|B_1f\|^2_{H^s}+
C\|Bu\|^2_{H^s}+C\|B_1u\|^2_{H^{s-\eps_0}}+C\|u\|^2_{H^{s-(r-1)+\eps_0}}.\]
We can then iterate this argument just as step 5 in the proof of \cite[Theorem E.47]{DZ19}, and improve the $\|B_1u\|^2_{H^{s-\eps_0}}$ term to $\|B_1u\|^2_{H^{s-\ell \eps_0}}$ for all $\ell$ large, we then obtain \eqref{estimeeProp} for $u\in C^\infty(M)$. Using the regularization argument \cite[Lemma E.45]{DZ19}, \eqref{estimeeProp} also holds for $u\in H^{s}(M)$ so that $P^\sharp u\in H^s(M)$.
To extend to general $u\in H^{s-(r-1)+\eps_0}(M)$ with $B_1Pu\in H^s$ and $Bu\in H^s$, one can use the regularization procedure explained in \cite[Section E.7, Exercices 10 and 31]{DZ19} which adapts exactly in the same way to our setting: set $N=s+(r-1)+\eps_0$, we use the regularization operator $X_\epsilon:=\mathrm{Op}(\ls\epsilon\xi\rs^{-N})$
and an elliptic approximate inverse~$Y_\epsilon$ so that $X_\eps Y_\eps=1+\mc{O}_{\Psi^{-\infty}(M)}(\eps^\infty)$; as operators in $\Psi^0(M)$, they are bounded uniformly in~$\epsilon$; 
for $r>0$, one can use Proposition \ref{composition_mfd} to get 
\begin{equation}\label{Peps}
P^\sharp_\epsilon:=X_\epsilon P^\sharp Y_\epsilon=P^\sharp+i\mathrm{Op}(\ls\epsilon\xi\rs^{N}\{ p_1,\ls\epsilon\xi\rs^{-N}\})+\mc{O}_{\tilde\Psi_{1,1}^{-(r-1)}}(1),\end{equation}
as in the classical case. The symbol $\ls\epsilon\xi\rs^{N}\{ p_1,\ls\epsilon\xi\rs^{-N)}\}$ is uniformly bounded in~ $S^{0}_{1,1}(M)$, the region in $\bbar{T}^*M$ 
where $P^\sharp_\eps$ and $P^\sharp$ are not microlocally in $\tilde{\Psi}^{1-r}_{1,1}(M)$ are the same (thus uniform in $\eps$), and~$u\in H^{-N}(M)$ is in~$H^{s}(M)$ if~$X_\epsilon u$ is uniformly bounded in~$H^s(M)$. Using this, the proof above extends to give 
\[\|X_\eps Au\|_{L^2}^2\leq C\|X_\eps B_1f\|^2_{H^s}+
C\|X_\eps Bu\|^2_{H^s}+C\|u\|^2_{H^{-N}}\]
uniformly in $\eps>0$, and by letting $\eps\to 0$ we obtain the result.
\end{proof}
\begin{rem} This result also extends in the obvious way if $P\in \prescript{r}{}{\rm Diff}^1(M;E)$ acts on a bundle $E$ and the principal symbol of $P$ is of the form $p_1 \otimes {\rm Id}$ for some $p_1\in C^r(T^*M)$ as above.
\end{rem}

The second type of propagation estimates are the radial estimates. These were introduced by Melrose in scattering theory \cite{Melrose1994} and developped by Dyatlov-Zworski \cite{DZ16a,DZ19} for smooth Anosov flows. We will now explain how to modify their proof to adapt them to paradifferential operators. We notice that for smooth generators of hyperbolic flows, 
radial source/sink estimates are also proved by Guedes Bonthonneau-Lefeuvre in H\"older spaces 
\cite{Bonthonneau-Lefeuvre}.

We assume that~$P$ is as in Proposition \ref{propofsing}. The principal symbol will be~$\sigma(P)=p_1$, real valued and homogeneous of degree $1$, and we denote by $\Phi_t=e^{tH_{p_1}}$ the Hamilton flow acting on the fiber-radial compactification of the cotangent bundle $\bbar{T}^*M$.

A \emph{radial source} is a nonempty compact $\Phi_t$-invariant set 
\[L\subset\lB\ls\xi\rs^{-1}p=0\rB\cap\partial \overline T^*M,\]
such that for some neighborhood $U\subset \bbar{T}^*M$ 
of $L$, there is $C,\nu>0$ such that uniformly in $(x,\xi)\in U$, 
\[ \lim_{t\to -\infty}{\rm dist}(\kappa(\Phi_t(x,\xi)), L)=0, \quad  \forall t\leq 0 ,\quad |\Phi_t(x,\xi)|\geq Ce^{\nu |t|}|\xi|;\]
here $\kappa: T^*M\to \pl\bbar{T}^*M$ is the canonical projection. 
A \emph{radial sink} is a radial source for $-p_1$.

A function $a\in C^0(\bbar{T}^*M)$ is called \emph{eventually positive on}~$L$ if $\exists T>0$ such that
\[\int_0^T a\circ\Phi_t\d t>0\quad\text{on }L.\]
A symbol is said to be eventually negative if $-a$ is eventually positive.

\begin{lemm}\label{approximationb}
Let $P$ be as in Proposition \ref{propofsing} with principal symbol $\sigma(P)=p_1$.
Let $a\in \prescript{s}{}{\tilde S}^{0}_{1,1}(M)$ for $s\in (0,1)$ such that that there exists $a_0\in C^{s}S^0(M)$, homogeneous of degree $0$, with $a-a_0\in C^{s}S_{1,1}^{-s}(M)$. Assume that  $a_0$, and thus $a$, is eventually positive on the radial source $L$.
 Then, there exists $b\in S^0(M)$ homogeneous of degree $0$ for large 
$|\xi|$ such that $H_{p_1}b+a>0$ on $L$. 
\end{lemm}
\begin{proof} 
We can take a smooth $a_0'\in S^0(M)$ 
so that $\|a_0'-a_0\|_{C^0}<\eps$ for $\eps>0$ small.
Let us then define in the region $\{|\xi|>1\}$
\begin{equation}\label{defofb}
b_0:=\frac1T\int_0^T(T-t) a_0'\circ\Phi_t\d t\in C^{r-1}S^0(M)
\end{equation} 
which is homogeneous of degree $0$, and we have $H_{p_1}b_0+a_0'=\frac{1}{T}\int_0^Ta'_0\circ \Phi_t \d t \in C^{r-1}S^0(M)$ which is well-defined in $\{|\xi|>1\}$. Moreover
one has on $L$
\[a_0+H_{p_1}b_0>a_0'+H_{p_1}b_0-\eps>\frac{1}{T}\int_0^Ta_0\circ \Phi_t \d t -2\eps>\eps
\] 
if $\eps<\frac{1}{3T}\int_0^Ta_0\circ \Phi_t \d t$. Using Lemma \ref{approximationlemma}, there is $b\in S^0(M)$ so that $\|H_{p_1}b-H_{p_1}b_0\|_{L^\infty}<\eps$, which ends the proof.
\end{proof}

\begin{prop}\label{sourceestimate}
    Let $P={\rm Op}(p)\in \prescript{r}{}{\rm Diff}^1(M)$ with~$r>1$ and assume that $p_1:=\sigma(P)$ is real valued. Let $\Phi_t=e^{tH_{p_1}}: T^*M\rightarrow T^*M$
be the flow of the Hamilton vector field of $p_1$  and assume $L\subset \pl\bbar{T}^*M$ 
   is a radial source for~$p_1$.
 Let $P^\sharp:=T_p\in \prescript{r}{}{\tilde \Psi}_{1,1}^1(M)$ be its paradifferential operator. Let $s_0\in\R$ satisfying the following \emph{threshold condition}: 
   \begin{equation}\label{symb<0}
   \sigma(\Im P)+ s_0\frac{H_{p_1}\ls\xi\rs}{\ls\xi\rs} \textrm{ is eventually negative on }L.
   \end{equation}
   If $s\geq s_0$, then for any~$B_1\in \Psi^0(M)$ such that~$L\subset\mathrm{ell}(B_1)$, there exists~$A\in \Psi^0(M)$ with~$L\subset\mathrm{ell}(A)$, such that for all~$\eps>0$ there is $C>0$ such that for all $u\in H^{s-(r-1)+\eps}(M)$ with $B_1u\in H^{s_0}(M)$, $B_1P^\sharp u\in H^s(M)$, then $Au\in H^s(M)$ and 
   \begin{equation}\label{radialpointest1}
   \| Au\|_{H^s(M)}\leq C\| B_1P^\sharp u\|_{H^s(M)}+C\| u\|_{H^{s-(r-1)+\eps}(M)}.
   \end{equation}
\end{prop}
\begin{proof}
Now that we have all the paradifferential calculus properties at hand, the proof is essentially the same as \cite[Theorem E.52]{DZ19}. We first prove \eqref{radialpointest1} for $u\in C^\infty(M)$. 
Using that $L$ is a source, observe that, if \eqref{symb<0} is true, then  it is also true by replacing $s_0$ by any $s\geq s_0$. Notice that $\Im P^\sharp\in \prescript{r-1}{}{\tilde \Psi}^{0}_{1,1}(M)$ using Proposition \ref{composePara}, and its principal symbol 
$\sigma(\Im P^\sharp)$ can be obtained using Proposition \ref{expansionsymbol} and the expression \eqref{formulapsharp} for $p^\sharp$: we have
\[\sigma(\Im P^\sharp)-\sigma(\Im P)\in C^{r-1}S_{1,1}^{-(r-1)}(M).\]
Thus by Lemma \ref{approximationb} and \eqref{symb<0}, there is $b\in S^{0}(M)$ homogeneous of degree $0$ for large $|\xi|$ so that $\sigma(\Im P^\sharp)+s\frac{H_{p_1}\ls\xi\rs}{\ls\xi\rs}-H_{p_1}b<0$.
Take $U\subset \bbar{T}^*M$ a neighborhood of $L$  such that
\[\exists \delta>0 \textrm{ small }, \quad \sigma(\Im P)+s\frac{H_{p_1}\ls\xi\rs}{\ls\xi\rs}-H_{p_1}b<-2\delta \quad  \textrm{ in }U.
\]
Then, as in \cite[Lemma E.53]{DZ19}, we claim that for each $\eps>0$ small, there is $\chi\in C_c^\infty(U;\R^+)$ homogeneous of degree $0$ near $\pl\bbar{T}^*M$ so that $\chi>0$ on $L$ and $H_{p_1}\chi \leq \eps$ and $\supp(H_{p_1}\chi)\cap L=\emptyset$. To obtain this function, we shrink $U$ so that ${\rm dist}(\kappa(\Phi_t(x,\xi)), L)\to 0$ as $t\to -\infty$ uniformly, let $\psi\in C_c^\infty(U,[0,1])$ equal to $1$ near $L$ and homogeneous of degree $0$ for $|\xi|$ large, then set 
$\chi_0:=T^{-1}\int_T^{2T} \psi\circ \Phi_{-t} \d t$ where $T>0$ is large enough so that $\Phi_{-t}({\rm supp}(\psi))\subset \{\psi=1\}$ for all $t\geq T$. One has $\chi_0\geq 0$, $\chi_0=1$ near $L$ (thus $H_{p_1}\chi_0=0$ near $L$), $\chi_0$ homogeneous of degree $0$ for large $|\xi|$,
$H_{p_1}\chi_0\leq 0$ and $\chi_0>0$ on $L$, but $\chi_0$ is only $C^{r-1}$ in the $x$-variable.
Using Lemma \ref{approximationlemma}, we find, for any $\eps>0$ small, a function $\chi_\eps\in S^0(M)$ non-negative so that $\chi_\eps =1$ in an $\eps$-independent neighborhood of $L$, and as $\eps>0$ goes to $0$, $\|\chi_\eps-\chi_0\|_{C^0}\leq \eps$ and $\|H_{p_1}(\chi_\eps-\chi_0)\|_{L^\infty}\leq \eps$. Thus $H_{p_1}(\chi_\eps)\leq \eps$ and $\supp(H_{p_1}\chi_\eps)$ is contained in a fixed compact set of $\bbar{T}^*M$ (uniform in $\eps>0$) not intersecting $L$ .

We can then define $g_\eps:=e^{-b}\chi_\eps \in C_c^\infty(U)\cap S^0(M)$ homogeneous of degree $0$ for $|\xi|$ large, and remark that 
\begin{equation}\label{Hpgdelta} 
g_\eps H_{p_1}g_\eps+\sigma(\Im P)g_\eps^2 +s\frac{H_{p_1}\cjg \xi\cjd}{\cjg \xi\cjd}g^2_\eps \leq -\delta g^2_\eps + e^{-2b}\chi_\eps H_{p_1}(\chi_\eps) \leq -\delta g_\eps^2+\eps \tilde{\chi}^2
\end{equation}
where $\tilde{\chi}\in C_c^\infty(U,\rr^+)$ is supported far from $L$ and independent of $\eps>0$.
Then we claim that if $A,B_2\in \Psi^0(M)$ are such that $g>0$ on ${\rm WF}(A)$, $B_2$ equal to $1$ microlocally near $L$ and ${\rm supp}(g)\cap \WF(1-B_2)=\emptyset$, for all $\eps_0>0$ small there is $C>0$ such that for all $u\in C^\infty(M)$, if $f:=P^\sharp u$,
\begin{equation}\label{1stboundtoprove}
\|Au\|_{H^s(M)}\leq C (\|B_2f\|_{H^s(M)}+\|B_2u\|_{H^{s-\eps_0}(M)}+\|u\|_{H^{s-(r-1)+\eps_0}(M)})
\end{equation}
To prove this, we proceed as in the proof of Proposition \ref{propofsing}. Let $G_\eps={\rm Op}(\cjg \xi\cjd^{s}g_\eps)\in \Psi^{2s}(M)$. Then $\cjg f,G_\eps^*G_\eps u\cjd  = \Re \cjg Zu,u\cjd$ with $Z:=\frac{i}{2}[{\Re P^\sharp},G_\eps^*G_\eps]+G_\eps^*G_\eps\Im P^\sharp \in \prescript{r-1}{}{\tilde \Psi}^{2s}_{1,1}(M)$ and $Z$ has principal symbol (using $r>1$ and Proposition \ref{composition_mfd})
\[ \sigma(Z)=\cjg \xi\cjd^{2s}\Big(g_\eps H_{p_1}g_\eps + g_\eps ^2\Big(s\frac{H_{p_1}\cjg \xi\cjd}{\cjg \xi\cjd} + \sigma(\Im P^\sharp)\Big)\Big)+o(|\xi|^{2s}) \textrm{ as }|\xi|\to \infty.\]
 By \eqref{Hpgdelta}, $\sigma(Z+\delta G_\eps^*G_\eps -\eps {\rm Op}(\tilde{\chi} \cjg \xi\cjd^{s})^*{\rm Op}(\tilde{\chi}\cjg \xi\cjd^{s}))\leq 0$, thus by Proposition \ref{Garding} (as in the proof of Proposition \ref{propofsing}), there is $C>0$ (independent of $\eps$) and $C_\eps>0$ depending on $\eps>0$ so that for each $u\in C^\infty$
\[ \Re \cjg Zu,u\cjd \leq -\delta\|G_\eps u\|^2_{L^2}+C\eps\|{\rm Op}(\tilde{\chi})u\|^2_{H^s}+C_\eps \|B_2u\|^2_{H^{s-\eps_0}}+C_\eps \|u\|^2_{H^{s-(r-1)+\eps_0}}.\]
Then as in the proof of Proposition \ref{propofsing} we get by Cauchy-Schwartz
\[ 
\delta \|G_\eps u\|_{L^2}^2 \leq C_\eps \|B_2f\|^2_{H^s}+C\eps\|{\rm Op}(\tilde{\chi})u\|^2_{H^s}+C_\eps\|B_2u\|^2_{H^{s-\eps_0}}+C_\eps \|u\|^2_{H^{s-(r-1)+\eps_0}}
\]
for some $C>0$ uniform in $\eps>0$. Since $g_\eps=\cjg \xi\cjd^s e^{-b}$ in an $\eps$-independent neighborhood  of $L$, we can find $A={\rm Op}(a)\in \Psi^{s}(M)$ independent of $\eps>0$ with $a\in C_c^\infty(U)$ 
such that $g_\eps\geq \cjg \xi\cjd^s a\geq c_0\cjg \xi\cjd^{s}$ near $L$ for some $c_0>0$. 
This implies, by sharp G\aa rding applied to $G_\eps^* G_\eps-{\rm Op}(\cjg \xi\cjd^s a)^*{\rm Op}(\cjg \xi\cjd^s a)$, that there is $C>0$ independent of $\eps>0$ and $C_\eps>0$ depending on $\eps>0$ such that
$\|G_\eps u\|_{L^2}^2\geq C^{-1}\|Au\|^2_{H^s}-C_\eps \|B_2u\|_{H^{s-\eps_0}}^2-C_\eps\|u\|^2_{H^{s-(r-1)+\eps_0}}$. Therefore we obtain
\begin{equation}\label{firststopradial}
\delta \|A u\|_{H^s}^2 \leq C_\eps \|B_2f\|^2_{H^s}+C\eps\|{\rm Op}(\tilde{\chi})u\|^2_{H^s}+C_\eps\|B_2u\|^2_{H^{s-\eps_0}}+C_\eps \|u\|^2_{H^{s-(r-1)+\eps_0}}.
\end{equation}

Since $\supp(\tilde{\chi})\cap L=\emptyset$ and $L$ is a source, there is a uniform time $T'$ so that for each $(x,\xi)\in \supp(\tilde{\chi})$, $\cup_{t\in [0,T]}\Phi_{-t}(x,\xi) \cap \WF(A)\not=\emptyset$. We can then apply Proposition \ref{propofsing} with $(A,B)$ replaced by $({\rm Op}(\tilde{\chi}),A)$: this gives for some $C>0$ uniform in $\eps>0$
\[\|{\rm Op}(\tilde{\chi})u\|^2_{H^s}\leq C\|Au\|^2_{H^s}+ C\|B_1f\|^2_{H^s}+C\|u\|^2_{H^{s-(r-1)+\eps_0}}.\] 
Combining with \eqref{firststopradial}, this leads, by fixing $\eps\ll \delta$, to the bound
\eqref{1stboundtoprove}. 
Then, the argument of step 2 in \cite[Theorem E.52]{DZ19}, based on applying Proposition \ref{propofsing} to remove the $\|B_2u\|^2_{H^{s-\eps_0}}$ term, can be applied verbatim in our case
and yields the estimate \eqref{radialpointest1} for all $u\in C^\infty(M)$. Using the regularization argument \cite[Lemma E.45]{DZ19}, \eqref{radialpointest1} also holds for $u\in H^{s}(M)$ so that $P^\sharp u\in H^s(M)$.

Finally, to obtain the result for $u\in H^{s-(r-1)+\eps_0}(M)$ with  $B_1u\in H^{s_0}(M)$, $B_1f\in H^s(M)$
set $G={\rm Op}(\cjg \xi\cjd^{s})$, we apply the same argument as in \cite[Section E.7, Exercise 35]{DZ19}. First we regularize using $X_\epsilon=\mathrm{Op}(\ls\epsilon\xi\rs^{-(s-s_0)})$ for $\eps>0$ small
and its approximate inverse~$Y_\epsilon$ satisfying $X_\eps Y_\eps=1+\mc{O}_{\Psi^{-\infty}(M)}(\eps^\infty)$; as operator in $\Psi^0(M)$ (resp. $\Psi^{(s-s_0)}(M)$), $X_\eps$ (resp. $Y_\eps$) is bounded uniformly in~$\epsilon$. We have for all $\eps>0$ small (as \eqref{Peps})
\[ P_\epsilon^\sharp:=X_\epsilon P^\sharp Y_\epsilon= P^\sharp +i\mathrm{Op}\lp\ls\epsilon\xi\rs^{s-s_0}H_{p_1}(\ls\epsilon\xi\rs^{-(s-s_0)})\rp+\mc{O}_{\tilde\Psi_{1,1}^{-(r-1)}}(1).\]
The symbol 
\begin{equation}\label{subprincipalterm}
\ls\epsilon\xi\rs^{(s-s_0)}H_{\sigma(P^\sharp)}(\ls\epsilon\xi\rs^{-(s-s_0)})=-(s-s_0)\frac{\eps^2 \cjg \xi\cjd^2}{\cjg \eps \xi\cjd^2}\frac{H_{\sigma(P^\sharp)}\cjg \xi \cjd}{\cjg \xi\cjd}
\end{equation}
is uniformly bounded in~ $\prescript{r-1}{}{\tilde S}^{0}_{1,1}(M)$ in $\eps$, and ~$u\in H^{s_0}(M)$ belongs to $H^s(M)$ if~$X_\epsilon u$ is uniformly bounded in~$H^s(M)$.  It suffices to show the estimate
\begin{equation}\label{esttoprove}
\lA X_\epsilon Au\rA_{H^s(M)}\leq C\lA X_\epsilon B_1f\rA_{H^{s}(M)}+C\lA u\rA_{H^{s-(r-1)+\eps_0}(M)},
\end{equation}
which is a consequence of proving the radial estimates for~$P^\sharp_\epsilon$ and~$u\in C^\infty(M)$ uniformly in $\eps>0$. Using \eqref{subprincipalterm}, one sees that~$P^\sharp_\epsilon$ satisfies the threshold condition uniformly in $\eps>0$, i.e. $\int_0^T(\sigma(\Im P^\sharp_\eps)+ s\frac{H_{p_1}\ls\xi\rs}{\ls\xi\rs})\circ \phi_t \d t<0$ for some $T>0$ independent of $\eps>0$. The proof of \eqref{radialpointest1} then applies for $P_\eps^\sharp$ replacing $P^\sharp$ and $X_\eps u$ replacing $u$, if with $u\in H^{s-(r-1)+\eps_0}(M)$: 
\[\lA  A X_\epsilon u\rA_{H^s(M)}\leq C\lA B_1X_\eps f\rA_{H^{s}(M)}+C\lA X_\eps u\rA_{H^{s-(r-1)+\eps_0}(M)}+o(1)\|u\|_{H^{s-(r-1)+\eps_0}}\]
 as $\eps\to 0$. 
Then one uses that, if $Q={\rm Op}(q)\in \Psi^0(M)$ and  $\tilde{Q}\in \Psi^0(M)$ elliptic on ${\rm WF}(Q)$, 
\[\begin{split}
QX_\eps=& X_\eps Q -[i{\rm Op}(\ls\epsilon\xi\rs^{s-s_0}\{ q,\ls\epsilon\xi\rs^{-(s-s_0)}\})+\mc{O}_{\Psi^{-2}(M)}(1)]X_\eps\tilde{Q}+\mc{O}_{\Psi^{-\infty}(M)}(1)\\
=& X_\eps Q+\mc{O}_{\Psi^{-1}(M)}(1)X_\eps\tilde{Q}+\mc{O}_{\Psi^{-\infty}(M)}(1)
\end{split}\]
so that for all $u\in H^{s-(r-1)+\eps_0}(M)$, $|\|QX_\eps u\|_{H^s(M)}-\|X_\eps Qu\|_{H^s(M)}|\leq C\|X_\eps \tilde{Q}u\|_{H^{s-1}(M)}+C\|u\|_{H^{s-(r-1)+\eps_0}}$ for some uniform $C>0$. Applying this to $Q=A$ and $Q=B_1$,  we get for each $B_1'\in \Psi^0(M)$ elliptic near $B_1$
\[\lA  X_\epsilon Au\rA_{H^s(M)}\leq C\lA X_\eps B_1'f\rA_{H^{s}(M)}+C\|X_\eps B_1u\|_{H^{s-1}(M)}
+C\lA u\rA_{H^{s-(r-1)+\eps_0}(M)}\]
Then, we apply again the argument of step 2 in \cite[Theorem E.52]{DZ19} (based on applying Proposition \ref{propofsing}) to remove the $\|X_\eps B_1u\|_{H^{s-1}}$ term, and we finally reach the estimate \eqref{esttoprove} with $B_1'$ instead of $B_1$.
\end{proof}

\begin{rem} Here again, the result also holds for operators acting on bundles, provided the principal symbol is of the form ${\sigma}({\rm Re}\, P)(x,\xi)=p_1(x,\xi)\otimes {\rm Id}$
\end{rem}

In the same way, we can prove the following:
\begin{prop}\label{sinkestimate}
     Let $P={\rm Op}(p)\in \prescript{r}{}{\rm Diff}^1(M)$ with~$r>1$ and assume that $p_1:=\sigma(P)$ is real-valued. Let $\Phi_t=e^{tH_{p_1}}: T^*M\rightarrow T^*M$ be the flow of the Hamilton vector field of $p_1$  and assume $L\subset \pl\bbar{T}^*M$ 
   is a radial sink for~$p_1$.
 Let $P^\sharp:=T_p\in \prescript{r}{}{\tilde \Psi}_{1,1}^1(M)$ be its paradifferential operator. 
 Assume that there exists ~$s_0\in\R$ satisfying the following \emph{threshold condition}:
\[\sigma(\Im P)+s_0\frac{H_{p_1}\ls\xi\rs}{\ls\xi\rs} \textrm{
is eventually negative on }L.\]
Then for any~$B_1\in \Psi^0(M)$ such that~$L\subset\mathrm{ell}(B_1)$, there exists~$A,B\in \Psi^{0}(M)$ such that~$L\subset\mathrm{ell}(A)$, $\mathrm{WF}(B)\subset\mathrm{ell}(B_1)\setminus L$, such that for all~$\eps_0>0$, there is $C>0$ so that for all $u\in H^{s-(r-1)+\eps_0}(M)$ such that $Bu\in H^s(M)$ and  $B_1P^\sharp u\in H^{s}(M)$, then $Au\in H^s(M)$ with
   \[\lA Au\rA_{H^s(M)}\leq C\|Bu\|_{H^s(M)}+C\lA B_1P^\sharp u\rA_{H^s(M)}+C\lA u\rA_{H^{s-(r-1)+\eps_0}(M)}.\]
\end{prop}
\begin{proof} The proof follows closely the line of \cite[Theorem E.54]{DZ19}.
By Lemma \ref{approximationb} there is $b\in S^{0}(M)$ homogeneous of degree $0$ for large $|\xi|$ so that $\sigma(\Im P^\sharp)+s\frac{H_{p_1}\ls\xi\rs}{\ls\xi\rs}-H_{p_1}b<0$.
Take $U\subset \bbar{T}^*M$ a neighborhood of $L$  such that ${\rm ell}(B_1)\subset U$ and 
\[\exists \delta>0, \quad \sigma(\Im P)+s\frac{H_{p_1}\ls\xi\rs}{\ls\xi\rs}-H_{p_1}b<-2\delta \quad  \textrm{ in }U.
\]
Let $\chi_1,\chi_2\in C_c^\infty(U)$ such that $\chi_1=1$ on ${\rm supp}(\chi_2)$ and $\chi_2=1$ on $L$, and let $\psi\in C_c^\infty(U\setminus L)$ equal to $1$ on ${\rm supp}(\chi_1)\cap {\rm supp}(1-\chi_2)$. Let $g:=e^{-b}\chi_1$, $A={\rm Op}(\chi_2)$, $B={\rm Op}(\psi)$, $B_2=A+B$ and $G={\rm Op}(\cjg \xi\cjd^sg)$.  Then as in Proposition \ref{sourceestimate}, if $f:=P^\sharp u$ for $u\in C^\infty(M)$, one has $\cjg f,G^*Gu\cjd  = \Re \cjg Zu,u\cjd$ with $Z:=\frac{i}{2}[{\Re P^\sharp},G^*G]+G^*G\Im P^\sharp \in \prescript{r-1}{}{\tilde \Psi}_{1,1}^{2s}(M)$ and  
\[ 
\sigma(Z-C_0^2{B'}^*B'+\delta G^*G)=\cjg \xi\cjd^{2s}(g H_{p_1}g+\sigma(\Im P)g^2+s\frac{H_{p_1}\cjg \xi\cjd}{\cjg \xi\cjd}g^2)+o(|\xi|^{2s})\leq 0
\]
for $|\xi|$ large, where $B':={\rm Op}(\cjg \xi \cjd^s\psi)$ and $C_0\psi\geq H_p(\chi_1)$. Applying Proposition \ref{Garding} to $Z-C_0^2{B'}^*B'+\delta G^*G$, there is $C>0$ such that for all $u\in C^\infty(M)$
\[ \Re \cjg Zu,u\cjd \leq -\delta\|Gu\|^2_{L^2}+C\|Bu\|^2_{H^s}+C\|B_2u\|^2_{H^{s-\eps_0}}+C\|u\|^2_{H^{s-(r-1)+\eps_0}}.\]
Proceeding as in the proof of Proposition \ref{propofsing}, we obtain
\[ \begin{split}
\|Au\|_{H^s}\leq & C\|Bu\|_{H^s}+ C\|B_1f\|_{H^s}+C\|B_2u\|_{H^{s-\eps_0}}+C\|u\|_{H^{s-(r-1)+\eps_0}}\\
\leq & C\|Bu\|_{H^s}+ C\|B_1f\|_{H^s}+C\|Au\|_{H^{s-\eps_0}}+C\|u\|_{H^{s-(r-1)+\eps_0}}.
\end{split}\]
Then an interpolation estimate, exactly as in \cite[Theorem E.54]{DZ19}, allows to absorb 
the $\|Au\|_{H^{s-\eps_0}}$ term. To obtain the result for $u\in H^{s-(r-1)+\eps_0}(M)$ so that $Bu\in H^s(M)$ and $B_1f\in H^s(M)$, it suffices to apply the same regularization procedure explained in the proof of Proposition \ref{sourceestimate}. 
\end{proof}

\textbf{The case of operators acting on vector bundles.} 
Let us finally discuss briefly the case of an operator $P\in \prescript{r}{}{\rm Diff}^1(M;E)$ acting on a Hermitian vector bundle $(E,\cjg \cdot,\cdot\cjd_{E})$, with $P=-i\nabla_X+V$ where $V\in C^{r-1}(M;E\otimes E^*)$, $\nabla$ a smooth connection on $E$ and $X\in C^{r}(M;TM)$ a $C^r$ vector field. The principal symbol is 
$\sigma(P)=p_1(x,\xi)\otimes {\rm Id}$ where $p_1(x,\xi)=\xi(X(x))$ is real valued, scalar, and linear in $\xi$. In that case, let $\Phi_t$ be the Hamilton flow of $p_1(x,\xi)$ as before. The operator $P^*$ is defined using $\cjg \cdot,\cdot\cjd_{E}$ and a fixed smooth measure on $M$.
 Consider the 
pull-back bundle $\tilde{E}:=\pi^*E$ on $\bbar{T}^*M$ where $\pi:\bbar{T}^*M\to M$ is the projection on the base, and let $\tilde{\nabla}:=\pi^*\nabla$ be the pull-back connection. 
%We call $P_t:\tilde{E}_z\to \tilde{E}_{\Phi_t(z)}$ the parallel transport for $\tilde{\nabla}$ 
%along the curve $\cup_{s\in [0,t]}\phi_s(z)$. 
An element $a\in \tilde{E}\otimes \tilde{E}^*$ is said symmetric if for $v,v'\in \tilde{E}$, $\cjg a(v'),v\cjd_E=\cjg v',a(v)\cjd_E$, in that case $a$ can be identified with an element in the symmetric bundle $S^2\tilde{E}^*$. 

Then Proposition \ref{sourceestimate} and Proposition \ref{sinkestimate} also hold for such $P$, assuming that 
\[\sigma({\rm Im} P)+s_0\frac{H_{p_1}\cjg \xi\cjd}{\cjg \xi\cjd}{\rm Id}<0 \textrm{ on }L,\] 
when we view $\sigma({\rm Im} P)$ as a section of $S^2\tilde{E}^*$ using the Hermitian product $\cjg \cdot,\cdot\cjd_E$. Indeed, the proof of these proposition apply verbatim to this case (taking $b=0$ there), using that the sharp G\aa rding inequality also holds on bundles. 

Another case of interest is when $E=(\otimes^pTM) \otimes (\otimes^q T^*M)$ is the bundle of tensors. Then the same proof applies to operators of the form $i(\mc{L}_X+V)$ where $\mc{L}_X$ is the Lie derivative and $V$ a potential; indeed the principal symbol of $\mc{L}_X$ is still $p_1(x,\xi)=\xi(X)$.

\section{Regularity and rigidity of the stable/unstable foliation}

In this section, we will use the tools discussed in previous section to describe microlocal regularity and rigidity results on the stable/unstable foliation of Anosov flows. 

\subsection{Dimension 3} 
Let $M$ be a compact $3$-dimensional manifold and let $X$ be a smooth vector field 
generating an Anosov flow $\varphi_t$, with Anosov splitting $\R X\oplus E_u\oplus E_s$.
To make the arguments a bit more concrete, we will assume that $E_u,E_s$ are trivialisable\footnote{this assumption can easily be removed by working as in the higher dimensional case considered later, the important object being the projector $\pi_H=TM\to \mc{H}$ as in \eqref{piH} where $\mc{H}$ is a smooth one-dimensional bundle approximating $E_u$}. By \cite{HPS70}, the bundles  $E_u,E_s$ are $C^\gamma(M)$ for some $\gamma>0$. More precisely, (see \cite[Lemma 2.2]{FaureGuillarmou}) there exist vector 
fields $U_-,U_+\in C^\gamma(M;TM)$ so that $E_u=\rr U_-$ and $E_s=\rr U_+$. There are $C^\gamma(M)$ functions $r_\pm$ so that 
\begin{equation}\label{dphiU} 
\begin{gathered}
\, [X,U_-]=-r_- U_- ,\quad d\varphi_{-t}(x)U_-(x)=e^{-\int_{-t}^0 r_-(\varphi_s(x))ds}
U_-(\varphi_{-t}(x))\\
[X,U_+]=r_+ U_+ ,\quad d\varphi_{t}(x)U_+(x)=e^{-\int_{0}^t r_+(\varphi_s(x))ds}
U_+(\varphi_{t}(x)).\end{gathered}
\end{equation}
We consider two smooth non-vanishing vector fields $H,V$ such that 
$TM=\rr X\oplus \rr H\oplus \rr V$ and so that the induced projection on $H$ parallel to $\rr X\oplus \rr V$
\begin{equation}\label{piH} 
\pi_H: TM\to \rr H 
\end{equation}
is an isomorphism  $\pi_H: E_u\to \rr H$. For example we can consider $H$ to be a smooth approximation of $U_-$ and $V$ be a smooth approximation of $U_+$, satisfying 
\begin{equation}\label{approximationUpm}
\|U_--H\|_{C^\gamma}+\|U_+-V\|_{C^\gamma}<\eps ,\quad \|\mc{L}_XH-\mc{L}_XU_-\|_{C^0}+\|\mc{L}_XV-\mc{L}_XU_+\|_{C^0}<\eps
\end{equation} 
for some small $\eps>0$. Such approximation can be done  using a partition of unity and in local charts by convolutions, the first coordinate being chosen so that $\pl_{x_1}=X$ in that chart (as in the proof of Lemma \ref{approximationlemma}).
Up to multiplying $U_-$ by a positive $C^\gamma$ function $a$ satisfying $Xa\in C^\gamma$ and $\|a-1\|_{C^0}+\|Xa\|_{C^0}=\mc{O}(\eps)$ (in which case it simply amounts to modify $r_-$ by a coboundary $X\log(a)$) we can assume that
\[ U_-= H+r_VV+r_XX\]
for some $r=\left( \begin{array}{c}
r_V \\
r_X
\end{array}\right)\in C^\gamma(M;\rr^2)$. The functions $r_X,r_V$ characterize the regularity of the unstable foliation.
\begin{lemm}\label{rsolvesRiccati}
	The function $r$ solves a Riccati type equation 
	\begin{equation}\label{riccatiforr}
	0= -Xr(x)+B(x)r(x)+ (Q(x)r(x))r(x)+C(x)
	\end{equation}
	where $B,Q,C$ are smooth matrices and for each $\eps>0$, one can choose $H,V$ so that 
	\[
	|C(x)|+|Q(x)|<\eps, \quad \left |B(x)-\left( \begin{array}{ccc}
	-r_+-r_- & 0  \\
	0 & -r_-
	\end{array}\right)\right|<\eps.
	\]
	and $Q=(Q_0 \,\, 0)$ for $Q_0\in C^\infty(M)$.
	If $E_u\oplus E_s$ is smooth, as for example in the contact case where $E_u\oplus E_s=\mathrm{Ker}(\alpha)$ for some smooth contact form $\alpha$,  then we can choose $H,V\in E_u\oplus E_s$, and  then
	we have $r_X=0$ and
	\begin{equation}\label{riccaticontact}
	0= -Xr_V(x)+b(x)r_V(x)+ q(x)r_V(x)^2+c(x)
	\end{equation}
	with $|c(x)|+|q(x)|<\eps$ and $|b(x)+r_+(x)+r_-(x)|<\eps$.
\end{lemm}
\begin{proof}
	Let us write the matrix $A^t$ of 
	$d\varphi_t$ in the basis $(H,V,X)$ under the form 
	\[
	A^t=\left( \begin{array}{ccc}
	A_{11}^t & A_{12}^t & 0  \\
	A_{21}^t & A_{22}^t & 0 \\
	A_{31}^t & A_{32}^t & 1 
	\end{array}\right).
	\]
	The invariance $d\varphi_tU_-\in \rr U_-$ implies that 
    $U_-^t(x):=d\varphi_t(x)U_-(x)$ can be written under the form 
	\[U_-^t(x)=a_t(x)\Big(H(\varphi_t(x))+r_V(\varphi_t(x))V(\varphi_t(x))+r_X(\varphi_t(x))X(\varphi_t(x)\Big)\]
	for some non vanishing function $a_t(x)$.
	Using the matrix representation $A^t$ of $d\varphi_t$, it is direct to see that 
	\[ \Big( \begin{array}{c}
	r_V(\varphi_t(x)) \\
	r_X(\varphi_t(x))
	\end{array}\Big)=\frac{1}{A_{11}^t(x)+A_{12}^t(x)r_V(x)}\left( \begin{array}{c}
	A_{21}^t(x)+A_{22}^t(x)r_V(x)\\
	A_{31}^t(x)+A_{32}^t(x)r_V(x)+r_X(x)
	\end{array}\right).
	\]
	We differentiate this equation at $t=0$ and use that $d\varphi_0={\rm Id}$, we then 
	get the scalar Riccati type equation 
	\[
	0= -Xr(x)+B(x)r(x)+ (Q(x)r(x))r(x)+C(x)
	\]
	with 
	\[\begin{gathered}
	C(x)=\left( \begin{array}{c}
	\dot{A}_{21}(x)\\
	\dot{A}_{31}(x)
	\end{array}\right), \quad  
	B(x)= \left( \begin{array}{cc}
	\dot{A}_{22}(x)-\dot{A}_{11}(x) & 0\\
	\dot{A}_{32}(x) & -\dot{A}_{11}(x)
	\end{array}\right),\\
	 Q(x)=(
	-\dot{A}_{12}(x) \,\, 0)
	\end{gathered}\]
	and here we have set $\dot{A}_{ij}(x):=\pl_tA_{ij}^t(x)|_{t=0}$. In the case 
	where $E_u\oplus E_s$ is a smooth bundle, we can choose $H,V\in E_u\oplus E_s$ and then we have $r_X=A_{31}=A_{32}=0$, which gives the equation on $r_V$
	\[
	0= -Xr_V(x)+(\dot{A}_{22}(x)-\dot{A}_{11}(x))r_V(x)- \dot{A}_{12}(x)r_V(x)^2+\dot{A}_{21}(x).
	\]
	
	We notice that by taking $H$ a smooth approximation (at scale $\eps>0$) of $U_-$ and $V$ a smooth approximation of $U_+$, we have as $\eps\to 0$
	\[ \begin{gathered}
	A_{11}^t(x)=e^{\int_{0}^t r_-(\varphi_s(x))ds}+o(1),\quad  A_{22}^t(x)=e^{-\int_{0}^t r_+(\varphi_s(x))ds}+o(1)\\
	A_{12}^t(x)=o(1),
	\quad A_{21}^t(x)=o(1), \quad A_{31}^t(x)=o(1)
	\end{gathered}\]
	where the remainders $o(1)$ are uniform in $(x,t)$ for small $t$, as well as the $t$-derivative. 
	This implies, using \eqref{approximationUpm}, that with such a choice, 
	\[\begin{gathered} \dot{A}_{11}(x)=r_-(x)+o(1), \quad \dot{A}_{22}(x)=-r_+(x)+o(1), \\
	\dot{A}_{12}(x)=o(1),
	\quad \dot{A}_{21}(x)=o(1), \quad \dot{A}_{31}(x)=o(1)
	\end{gathered}\]
and that shows the desired result.
\end{proof}

\begin{rem} %Add refs
	In the geodesic flow case, it is known that the bundles~$E_u$ and~$E_s$ are trivial. Then taking for~$H$ the intersection of the horizontal bundle and of the kernel of the Liouville form, and for~$V$ the vertical bundle, the equation becomes the well-known Ricatti equation
	\[Xr+r^2+K=0\]
	where~$K$ is the lift of the Gauss curvature to the unit tangent bundle \cite{Paternain}.
\end{rem}

Define the subbundles $E_0^*,E_u^*$ and $E_s^*$ of $T^*M$ by 
\[E_u^*(E_u\oplus\R X)=0,\quad  E_s^*(E_s\oplus\R X)=0,\quad E_0^*(E_u\oplus E_s)=0.\]
Then we have the following result on the regularity of $E_u$ (the same obviously holds for $E_s$ by taking the flow in reverse time).
\begin{theo}\label{Thregulrigid1}
Let $X$ be a smooth vector field generating an Anosov flow $\varphi_t$ on a $3$-dimensional closed manifold.
Let $r\in C^\gamma(M,\rr^2)$ be a solution of the Riccati equation \eqref{riccatiforr} for $\gamma>0$ and $U_-$ be the unstable vector field of \eqref{dphiU}.\\
1) Then~${\rm WF}(r)\subset E_u^*$ and ${\rm WF}(U_-)\subset E_u^*$.\\ 
2) Assume that $X$ is preserving a smooth volume form. Then $r\in H^{1-\delta}(M),U_-\in H^{1-\delta}(M)$ for any~$\delta>0$. If $X$ is a contact flow, then $U_-\in H^{2-\delta}(M)$ for all $\delta>0$.\\ 
3) Assume that $X$  is preserving a smooth volume form. If $r\in H^{2+\delta}(M)$, or equivalently $U_-\in H^{2+\delta}(M)$, for some $\delta>0$, then $r\in C^\infty$ and  $U_-\in C^\infty$.\\
4) In general, let $s>0$ such that there is $T>0$ large so that uniformly on $M$
\[(s-\frac{3}{2}) \int_0^T r_+\circ \varphi_t \, dt> \tfrac{1}{2}\int_0^T r_-\circ \varphi_t\, dt.\]
If $U_-\in H^s(M)$, then $U_-\in C^\infty(M)$.\\

If the vector field $X$ is in $C^\beta(M)$ for some $\beta>1$, then the same properties hold but replacing ${\rm WF}$ by ${\rm WF}_{H^{\beta}}$ in 1), $U_-\in H^{2-\delta}(M)$ by $U_-\in H^{\min(\beta,2)-\delta}(M)$ in 2), and $U_-\in C^\infty$ by $U_-\in H^{\beta-\delta}(M)$ for all $\delta>0$ in 3) and 4).
\end{theo}
\emph{Remark.} The results in 3) are weaker than what is proved in \cite{Hurder-Katok,Hasselblatt1992PAMS} where H\"older (in fact Zygmund) regularity $C^{1-\delta}$ and $C^{2-\delta}$ are shown to hold. The rigidity result 4) 
is in Sobolev norms, thus  not contained in those references (recall that $H^{2+\delta}(M)$ is only included in $C^{3/2}(M)$ for $\delta\to 0$). There are however similar results for Sobolev spaces in the work of De la Llave \cite{DLL2001}.
\begin{proof}
First, as ~$r\in C^{\gamma}$, we also have~$r\in H^{\gamma'}(M)$ for any~$0<\gamma'<\gamma$. Using Proposition \ref{paraproduct}, 
the equation \eqref{riccatiforr} can then be paralinearized and rewritten as follows (with 
$r=(r_V\,\, r_X)^T$)
\begin{equation}\label{paradiffeq}
-iXr+iBr-2iQ_0T_{r_V} r_V=-iQ_0R(r_V,r_V)+iC
\end{equation}
where $T_{r_V}\in \prescript{\gamma}{}{\tilde \Psi}^{0}_{1,1}(M)$ is the paradifferential operator of $r_V$ and $R(r_V,r_V)\in H^{2\gamma'}(M)$, while $B,C,Q_0\in C^\infty$. Note that the right hand side then belongs to $H^{2\gamma'}(M)$, thus more regular in Sobolev scale than $r$. We denote also by 
$T_{r}$ the $\prescript{\gamma}{}{\tilde \Psi}^{0}_{1,1}(M)$ valued matrix 
\[T_r:= \left( \begin{array}{cc}
	Q_0T_{r_V} & 0\\
	0 & 0
	\end{array}\right) 
\textrm{ with  principal symbol } \sigma(T_r)=\left( \begin{array}{cc}
	Q_0r_V & 0\\
	0 & 0
	\end{array}\right) +o(1) \textrm{ as }|\xi|\to \infty.\]	
The operator~$P:=-iX+iB-2iT_{r} \in {\rm Diff}^1(M)+\prescript{\gamma}{}{\tilde \Psi}_{1,1}^{0}(M)\subset \prescript{1+\gamma}{}{\tilde \Psi}_{1,1}^{1}(M)$ has principal symbol given by ~$p_1(x,\xi):=\sigma(P)(x,\xi)=\xi(X)$. We can use the ellipticity estimates of Corollary \ref{estimeeell}: this gives that for each $Z\in \Psi^0(M)$ with $\WF(Z)\cap \{\xi(X)=0\}=\emptyset$, $Zr\in H^{2\gamma'+1}(M)$, that is 
$\WF_{H^{2\gamma'+1}}(r)\subset \{\xi(X)=0\}$. 
We fix a smooth measure $\mu$, for example using a Riemannian metric, this induces an $L^2$ scalar product.
The principal symbol of the imaginary part of $P$ is
\[\sigma(\Im P)=\frac{B+B^*}{2}-2\sigma(T_r)+\frac{{\rm div}(X)}{2} \]
where ${\rm div}(X)=\mc{L}_X\mu/\mu$ and we have used the measure $d\mu$ to define $L^2$-adjoints. 
By Lemma \ref{rsolvesRiccati} 
%and using that the divergence ${\rm div}(X)=0$ since we assume that $\varphi_t$ preserves a smooth measure $\mu$, 
we obtain 
\begin{equation}\label{sigmaImP}
\sigma(\Im P)=\left( \begin{array}{cc}
	-r_+-r_- & 0\\
	0 & -r_-
\end{array}\right)+\frac{{\rm div}(X)}{2}+\mc{O}(\eps)+o(1) \textrm{ as }|\xi|\to \infty. 
\end{equation}
Notice that in the case of a contact flow we can choose $r_-=r_+$, ${\rm div}(X)=0$ and $r_X=0$ so that we get a scalar equation and $\sigma(\Im P)=-2r_-+\mc{O}(\eps)+o(1)$ as $|\xi|\to \infty$. We take the canonical flat connection on the trivial $\rr^2$-bundle over $M$ and use the notations of Lemma \ref{sourceestimate}, in particular $\Phi_t=e^{tH_{p_1}}$ is the Hamilton flow of $p_1$, i.e. the symplectic lift of $\varphi_t$ in our case. Moreover $L:=E_s^*\cap \pl\bbar{T}^*M$ is a source for $\Phi_t$ since $\varphi_T$ is Anosov. By \eqref{dphiU}, on $E_s^*$ one has 
\[ \int_0^T \frac{H_{p_1}\cjg \xi\cjd}{\cjg\xi\cjd}\circ\Phi_t \, dt= \log \frac{\cjg \xi\circ \Phi_T\cjd}{\cjg \xi\cjd}= \log \frac{\cjg \xi\circ d\varphi_{T}^{-1}\cjd}{ \cjg \xi\cjd}=-\int_0^T r_-\circ \varphi_t\, dt.\]
Notice also that for $T>0$ large 
\[\int_0^T{\rm div}(X)\circ \varphi_t \, dt=\int_0^T (r_--r_+)\circ \varphi_t\, dt+o(T).\]
For $s\in\rr$, one has for $T>0$ that on $L$ 
\[\begin{gathered}
\int_0^T\lp\sigma(\Im P)+s \frac{H_p\ls\xi\rs}{\ls\xi\rs}\rp\circ\Phi_t dt=\\
\left( \begin{array}{cc}
	-\int_0^T (\tfrac{3}{2}r_++(\tfrac{1}{2}+s)r_-)\circ \varphi_t\, dt & 0\\
	0 & -\int_0^T ((\tfrac{1}{2}+s)r_-+\tfrac{1}{2}r_+)\circ \varphi_t\, dt
	\end{array}\right)+\mc{O}(\eps T).
\end{gathered}\]
We see that for all $s>-1/2$ this term is negative for $\eps>0$ chosen small enough. Thus we can apply Proposition \ref{sourceestimate} to deduce that there is $A\in \Psi^0(M)$ 
elliptic near $E_s^*$ such that $Ar\in H^{2\gamma'}(M)$ for all $\gamma'<\gamma$; we thus get $\WF_{H^{2\gamma'}}(r)\cap {\rm ell}(A)=\emptyset$ for all $\gamma'<\gamma$. Since, by the Anosov property of the flow, each $(x,\xi)\in T^*M\setminus E_u^*$ is such that there is $T$ so that $\Phi_{-T}(x,\xi)\in {\rm ell}(A)\cup\{\xi(X)\not=0\}$, we can use the propagation of singularity of Proposition \ref{propofsing} to deduce that for each $A'\in \Psi^0(M)$ with $\WF(A')\cap E_u^*=\emptyset$, then 
$A'r\in H^{2\gamma'}(M)$, i.e. $\WF_{H^{2\gamma'}}(r)\subset E_u^*$, $\forall \gamma'<\gamma$.

Now take $S\in \Psi^0(M)$ such that $\WF(S-1)\subset O_u$ and $\WF(S)\cap E_u^*=\emptyset$, where $O_u$ is an arbitrarily small conic neighborhood of $E_u^*$. Applying $S$ to \eqref{paradiffeq}, we get 
\[ PSr=-i(SQ_0R(r_V,r_V)+SC)+[P,S-1]r.\]
By applying Lemma \ref{WF(R)} with $a=b=r_V$, 
$\eps=\gamma$, $\alpha=\beta=\gamma'$ and $\delta=\gamma'$, using $\WF_{H^{2\gamma'}}(r)\subset E_u^*$   we see that $\WF_{H^{3\gamma'}}(Q_0SR(r_V,r_V))\subset E_u^*$ and we also have $\WF_{H^{3\gamma'}}([P,S-1]r)\subset O_u$ since $\WF(S-1)\subset O_u$.
We can then apply the same argument as above where now the right hand side is in $H^{3\gamma'}(M)$ microlocally outside $O_u$, and we obtain that $\WF_{H^{4\gamma'}}(Sr)\subset O_u$. Since
$O_u$ was chosen arbitrarily small, we conclude that $\WF_{H^{4\gamma'}}(r)\subset E_u^*$. Then we bootstrap this argument and obtain that 
\[ \WF(r)\subset E_u^*, \quad \WF(U_-)\subset E_u^*.\]

Next, we apply the sink radial estimate of Proposition \ref{sinkestimate}. We already know that 
$Ar\in H^{N}(M)$ for all $N\geq 0$ if $A\in \Psi^0(M)$ has $\WF(A)\cap E_u^*=\emptyset$.
As above, we have on $L'=E_u^*\cap \pl\bbar{T}^*M$.
For $s\in\rr$, one has for $T>0$ that on $L'$ 
\[\begin{gathered}
\int_0^T\lp\sigma(\Im\mathrm{P})+s \frac{H_{p_1}\ls\xi\rs}{\ls\xi\rs}\rp\circ\Phi_t\, dt=\\
\left( \begin{array}{cc}
	-\int_0^T (r_+(\tfrac{3}{2}-s)+\tfrac{1}{2}r_-)\circ \varphi_t\, dt & 0\\
	0 & -\int_0^T (\tfrac{1}{2}r_-+(\tfrac{1}{2}-s)r_+)\circ \varphi_t\, dt
	\end{array}\right)+\mc{O}(\eps T).
\end{gathered}\]
This is negative for $T$ large (and $\eps>0$ chosen small) if, uniformly on $M$,
\begin{equation}\label{condonthreshold} 
s \int_0^T r_+\circ \varphi_t \, dt< \frac{1}{2}\int_0^T(r_++r_-)\circ \varphi_t\, dt.
\end{equation}
For such $s$ we have that $r\in H^s(M)$ if $r\in H^{s-\gamma'}(M)$ for some $\gamma'<\gamma$. In the volume preserving case, the condition \eqref{condonthreshold} holds when $s<1$.
When the flow is contact $r_X=0$ and $r_+=r_-$ so that one has a single equation and the term above becomes negative when $s<2$, which shows that $r\in H^{2-\delta}(M)$ for all $\delta>0$.\\

To prove the rigidity result, we apply the source estimate to the flow in reverse time. Let now 
$P_-:=-P=iX-iB-2iT_{r}$ and we analyze the PDE $P_-r=iQ_0R(r_V,r_V)-iC$. The set $L'=E_u^*\cap \pl\bbar{T}^*M$ becomes a source for the Hamilton vector field $H_{-p_1}=-H_{p_1}$ where $-p_1=\sigma(P_-)=-\xi(X)$, whose flow is given by $\Phi_{-t}$. For $s\in \R$, let us analyze the negativity condition on $L'$ of 
\[\begin{gathered}
\int_0^T -\lp\sigma(\Im P)+s \frac{H_{p_1}\ls\xi\rs}{\ls\xi\rs}\rp\circ\Phi_{-t} dt=\\
\left( \begin{array}{cc}
	\int_{-T}^0 ((\tfrac{3}{2}-s)r_++\tfrac{1}{2}r_-)\circ \varphi_{t}\, dt & 0\\
	0 & \int_{-T}^0 (\tfrac{1}{2}-s)r_++\tfrac{1}{2}r_-)\circ \varphi_t\, dt
	\end{array}\right)+\mc{O}(\eps T).
\end{gathered}\]
Taking $\eps>0$ arbitrarily small and $T$ large enough, we see that if uniformly on $M$
\begin{equation}\label{conditiononsnonvol}
(s-\frac{3}{2}) \int_0^T r_+\circ \varphi_t \, dt> \frac{1}{2}\int_0^Tr_-\circ \varphi_t\, dt
\end{equation}
then we can apply Proposition \ref{sourceestimate} to deduce that if $r\in H^s$, then in fact 
$r$ has the regularity of the right hand side $iQ_0R(r_V,r_V)-iC\in H^{s+\gamma'}(M)$. Bootstrapping the argument we get that $r\in C^\infty(M)$. In particular, if $X$ preserves a smooth measure, the condition is always satisfied if $s>2$ since 
$\int_0^Tr_-\circ \varphi_t dt=\int_0^Tr_+\circ \varphi_t dt+o(T)$ for large $T$.

Let us briefly discuss the case where the vector field $X\in C^\alpha$ for $\alpha>1$ and assume $\gamma<\alpha$. Using Lemma \ref{regularization}, we can write $-iX=T_{p}+{\rm Op}(p^\flat)$ where $p(x,\xi)=\xi(X)\in C^\alpha S^1(M)$, $p^\flat\in C^{\alpha-1}S^{1-\alpha}_{1,1}(M)$ and $T_p\in \prescript{\alpha}{}{\tilde S}^1_{1,1}(M)$. By \eqref{boundroughcoef}, ${\rm Op}(p^\flat):H^{s+1-\alpha}(M)\to H^s(M)$ for all $s\in (0,\alpha)$. Let $\beta:=\min(\alpha,\gamma+1)$. The equation \eqref{paradiffeq} can be rewritten as 
\[T_pr+iBr-2iQ_0T_{r_V} r_V=-iQ_0R(r_V,r_V)+iC-{\rm Op}(p^\flat)r\in H^{\min(2\gamma',\gamma'+\alpha-1)}(M)\]
for all $\gamma'<\min(\gamma,1)$.
Using the ellipticity result from Corollary \ref{estimeeell}, $r$ is microlocally 
in $H^{2\gamma'}(M)$ outside $\xi(X)=0$ (we have used that $2\gamma'<\gamma'+\alpha)$. We can then apply Propositions \ref{sourceestimate}, \ref{propofsing}   as above, we obtain that $r$ is microlocally $H^{1-\delta}(M)$ outside $E_u^*$ for all $\delta>0$ in the volume preserving case. In the contact case that gives that $u$ is microlocally in $H^{\min(\alpha,2)-\delta}(M)$ 
for all $\delta>0$.
Using Proposition \ref{sinkestimate} as above, we also directly obtain that if $u\in H^{s}(M)$ for 
$s$ satisfying \eqref{conditiononsnonvol}, then $r\in H^{\alpha-\delta}(M)$ for all $\delta>0$. If the flow is volume preserving and $C^\alpha$ for $\alpha>2$, this shows that $r\in H^s(M)$ with $s>2$ implies $r\in H^{\alpha-\delta}(M)$ for all $\delta>0$.
\end{proof}

\subsection{Higher dimension}

Let $M$ be a smooth compact manifold of dimension $d$ and let $X$ be a smooth vector field generating an Anosov flow $\varphi_t$. We denote by $E_u$ and $E_s$ the unstable and stable bundles as above. They are $C^\gamma$ H\"older continuous for some $\gamma>0$, \cite{HPS70}. We denote by $n_u=\dim E_u$ and $n_s=\dim E_s$ so that $d=1+n_u+n_s$.

We can construct $H\subset TM$ a smooth subbundle of dimension $n_u$ which can be chosen arbitrarily close to $E_u$, viewed as points in the Grassmanian $\mc{G}_{n_u}(M)$ of $n_u$ dimensional subpaces in $TM$. 
Similarly let $V$ be a smooth approximation of the weak unstable bundle $E_{s,0}:=E_s\oplus \rr X$. Due to our choices, there are smooth projections 
\[ \pi_H : TM\to H ,\quad \pi_V : TM\to V\] 
induced by the decomposition $TM=H\oplus V$
and $\pi_H$ is an isomorphism when restricted to a neighborhood $W_{H}\subset \mc{G}_{n_u}(M)$ corresponding to vector subbundles close to $H$ (for example contained in an unstable cone of $E_u$). In particular 
\[\pi_H(E_u): E_u \to H\] 
is a $C^\gamma$ H\"older isomorphism of bundles, with a $C^\gamma$ H\"older inverse $\pi_H^{-1}$.
We have 
\begin{lemm}\label{End}
There is a one-to-one smooth correspondance between $C^\gamma$- subbundles of dimension $n_u$ near $E_u$ and the space $C^\gamma(M; \mc{L}(H,V))$ of $C^\gamma$ linear maps from $H$ to $V$, given by 
\[ \Psi: E\in W_H \mapsto \pi_V\circ \pi_{H}(E)^{-1}\]
\end{lemm}
\begin{proof} The map is smooth, and its inverse is given by 
the map which to $\hat{U}\in  C^\gamma(M; \mc{L}(H,V))$ assigns the bundle given by the graph of $\hat{U}$,
i.e. the image of the linear map 
\[H\to TM=H\oplus V, \quad h\mapsto h+\hat{U}h.\]
The regularity of $\Psi(E)$ is easily determined by considering a local basis of smooth sections of $H,V$ and a $C^\alpha$ sections of $E$, and then writing the  $\mc{L}(H,V)$ map $\pi_V\circ \pi_{H}(E)^{-1}$ as a matrix in these basis: if $(T_j)_{j=1,\dots,n_u}$ is a basis of $H$, $(T_j)_{j=n_u+1,\dots,d}$ is a basis of $V$, and $(U_j)_{j=1,\dots,n_u}$ is a $C^\alpha$ basis of $E$, then the matrix representing the map $\Psi(E)$ in the basis $(T_j)_j$ is given by 
\[ \hat{U}= U^V(U^H)^{-1}
\]    
where $U_j=\sum_{k=1}^{n_u} U^H_{kj}T_k+\sum_{k=n_u+1}^{d}U^V_{kj}T_{k}$.
\end{proof}

The flow $d\varphi_t: TM\to TM$ also acts on the Grassmannian 
$\mc{G}_{n_u}(M)$ by simply writing
\[ \Phi_t: \mc{G}_{n_u}(M)\to \mc{G}_{n_u}(M), \quad \Phi_t(E)=d\varphi_t(E).\]
We would like to describe this action in terms of the linear maps described in Lemma \ref{End}.

\begin{lemm}\label{actiononEnd}
The action of $\Psi \Phi_t \Psi^{-1}$ on $\mc{L}(H, V)$ can be described as follows: let 
\[ A(t)=\left(\begin{array}{cc}
A_1(t) & A_2(t) \\
A_3(t) & A_4(t)
\end{array}\right) 
\]
be the block decomposition of $d\varphi_t: H\oplus V\to H\oplus V$, then one has 
\begin{equation}\label{formevolU} 
\Psi \Phi_t \Psi^{-1} \hat{U}=(A_3(t)+A_4(t)\hat{U})(A_1(t)+A_2(t)\hat{U})^{-1}\end{equation}
\end{lemm}
\begin{proof}
Let us use the bases and notations of the proof of Lemma \ref{End}. For small $t$, we get 
for each $x\in M$
\[A(t)U_j(x)=\sum_{i=1}^{d}U_{kj}(x)A_{ik}(t;x)T_i(\varphi_t(x))\] 
where $A_{ik}(t;x)$ are the matrix components of $A(t)$ at the point $x$ in the basis 
$(T_j)_j$, and $U_{jk}$ are the matrix components of $(U^H,U^V)$ in $(T_j)_j$, i.e. for $a_k\in\rr$ 
\[A(t)(\sum_{k=1}^da_kT_k(x))=\sum_{i,k=1}^dA_{ik}(t;x)a_k T_i(\varphi_t(x)).\]
This implies that 
\[ \begin{split}
\Psi \Phi_t \Psi^{-1} \hat{U}=& (A_3(t)U^H +A_4(t)U^V)(A_1(t)U^H+A_2(t)U^V)^{-1}\\
 =& (A_3(t)+A_4(t)\hat{U})(A_1(t)+A_2(t)\hat{U})^{-1}.
\end{split}\] 
This concludes the proof.
\end{proof}
As a corollary, we shall describe the infinitesimal action as a Riccati operator. First, to make invariant sense of derivatives on $\mc{L}(H,V)=V\otimes H^*$, we fix a Riemannian metric $g$ so  on $TM$ (here we use $H^*$ for the dual of $H$, not the annihilator as we did for $E_u^*,E_s^*$ before). 
We recall that we can differentiate a section $v\otimes h^*\in V\otimes H^*$ using the Lie derivative by viewing it as an element in $TM\otimes T^*M$:
\[  \mc{L}_X (v\otimes h^*)_x=\pl_t\Big(d\varphi_{-t}(\varphi_t(x))v_{\varphi_t(x)}\otimes d\varphi_t(x)^\top h^*_{\varphi_t(x)}\Big)\Big|_{t=0}\in  TM\otimes T^*M. \]

As in dimension $3$, using a partition of unity, local charts and regularizations through convolutions with $\pl_{x_1}=X$ in those charts, for each $\eps>0$ we can choose $H,V$ (equivalently $\pi_H,\pi_V$) and $g$ so that (using a fixed background metric on $TM$ for defining the norms)
\begin{equation}\label{apprixamationpiH}
\begin{gathered} 
\|\pi_H-\pi_{E_{u}}\|_{C^0}+\|\pi_V-\pi_{E_{s0}}\|_{C^0}<\eps, \quad \|g\|_{C^0}+\|\mc{L}_Xg\|_{C^0}\leq 1\\
\|\mc{L}_X(\pi_H-\pi_{E_{u}})\|_{C^0}+\|\mc{L}_X(\pi_V-\pi_{E_{s0}})\|_{C^0}<\eps.
\end{gathered}
\end{equation}
where $\pi_{E_u}:TM\to E_u$ and $\pi_{E_{s0}}:TM\to E_s\oplus \rr X$ are the $C^\gamma$ projections induced by the decomposition $TM=E_u\oplus(E_s\oplus \rr X)$. 
We can also assume that $|g(X,v)|<\eps\|v\|_g$ for all $v\in E_s$, 
for some arbitrarily  small $\eps>0$ and that $g$ is close enough (in H\"older norm) 
to some adapted metric for $X$ so that 
\begin{equation}\label{adaptedmetric}
 \begin{split} 
\forall t\geq 0,\quad &  (1-\eps)e^{(-\nu_u^{\max}-\eps) t}\leq \|d\varphi_{-t}|_{E_u}\|_{g}\leq (1+\eps)e^{(-\nu_u^{\min}+\eps) t},\\
 \forall t\geq 0, \quad & (1-\eps)e^{(-\nu_s^{\max}-\eps) t}\leq \|d\varphi_{t}|_{E_s}\|_{g}\leq (1+\eps)e^{(-\nu_s^{\min}+\eps) t}
\end{split}
\end{equation} 
where $\nu_u^{\min}, \nu_u^{\max}$ are the minimal expansion rates of the flow on $E_u$ and similarly for $E_s$.

\begin{corr}
Let $\hat{U}\in C^\alpha(M;\mc{L}(H,V))$ be the representation of a vector subbundle of dimension $n_u$ close to $E_u$ in $\mc{G}_{n_u}(M)$. This subbundle is invariant by $d\varphi_t$ if and only if 
\begin{equation}\label{riccatinabla} 
\mc{L}_X\hat{U}+\hat{U}\dot{A}_2\hat{U} +\hat{U}\dot{A}_1-\dot{A}_4\hat{U}-\dot{A}_3=0
\end{equation} 
where 
\[\begin{gathered}
\dot{A}_3(x)=\pl_{t}(d\varphi_{-t}(\varphi_t(x))\pi_Vd\varphi_t(x)\pi_H)|_{t=0}, \quad  
\dot{A}_4(x)=\pl_{t}(d\varphi_{-t}(\varphi_t(x))\pi_Vd\varphi_t(x)\pi_V)|_{t=0}\\
\dot{A}_1(x)=\pl_{t}(d\varphi_{t}(x)^{-1}\pi_Hd\varphi_t(x)\pi_H|_{t=0}, \quad  
\dot{A}_2(x)=\pl_{t}(d\varphi_{t}(x)^{-1}\pi_Hd\varphi_t(x)\pi_V)|_{t=0}.
\end{gathered}\]
Moreover for each $\eps>0$, one can choose $H,V$ arbitrarily close to $E_u$ and $E_s\oplus \rr X$ so that uniformly on $M$ 
\[
 \sum_{j=1}^4\|\dot{A}_j\|_{C^0}\leq \eps. \]
 \end{corr}
\begin{proof} 
Let us define $\hat{U}^t:=\Psi \Phi_t \Psi^{-1} \hat{U}$.
The subbundle $\hat{U}$ is invariant if and only if $\hat{U}^t(\varphi_{-t}(x))=\hat{U}(x)$ for all $x\in M$ and all $t\in\rr$. 
 We have 
\[\begin{split} 
\mc{L}_X\hat{U}=& \pl_t \Big( d\varphi_{-t}(\varphi_t(\cdot))\hat{U}(\varphi_t(\cdot))d\varphi_t(\cdot) \Big)\Big|_{t=0}
= \pl_t \Big( (\tilde{A}_3(t)+\tilde{A}_4(t)\hat{U})(\tilde{A}_1(t)+\tilde{A}_2(t)\hat{U})^{-1} \Big)\Big|_{t=0}\\
=& -\hat{U}\dot{A}_2\hat{U} -\hat{U}\dot{A}_1+\dot{A}_4\hat{U}+\dot{A}_3
\end{split}\]
with $\dot{A}_j:=\pl_t\tilde{A}_j(t;x)|_{t=0}$ and 
\[\begin{gathered}
\tilde{A}_3(t;x)=d\varphi_{-t}(\varphi_t(x))\pi_Vd\varphi_t(x)\pi_H, \quad  
\tilde{A}_4(t;x)=d\varphi_{-t}(\varphi_t(x))\pi_Vd\varphi_t(x)\pi_V\\
\tilde{A}_1(t;x)=d\varphi_{t}(x)^{-1}\pi_Hd\varphi_t(x)\pi_H, \quad  
\tilde{A}_2(t;x)=d\varphi_{t}(x)^{-1}\pi_Hd\varphi_t(x)\pi_V.
\end{gathered}\]
Fianlly, we choose $H\oplus V$ to satisfy \eqref{apprixamationpiH}, and we obtain the desired properties $\|\dot{A}_j\|_{C^0}=\mc{O}(\eps)$ for all $j$.
\end{proof}

\begin{theo}\label{Thregulrigid2}
Let $X$ be a smooth vector field generating an Anosov flow $\varphi_t$ preserving a smooth measure on $M$.
Let $\hat{U}\in C^\gamma(M;\mc{L}(V\otimes H^*))$ be the section parametrizing the bundle $E_u$ and $\nu_u^{\min},\nu_u^{\max},\nu_s^{\min},\nu_s^{\max}$ the minimal/maximal expansion rates of the flow, see \eqref{adaptedmetric}.\\ 
1) Then its wavefront set satisfies $\WF(\hat{U})\subset E_u^*$ and $\hat{U}$ has Sobolev 
regularity $H^{s}(M)$ near $E_u^*$ for all $s<(\nu_{u}^{\min}+\nu_{s}^{\min})/\nu_{s}^{\max}$.\\ 
2) If $\hat{U}\in H^s(M)$, and thus $E_u\in H^s$, for $s>(\nu_{u}^{\max}+\nu^{\max}_s)/\nu_s^{\min}$, then $\hat{U}\in C^\infty$ and $E_u\in C^\infty$.\\

In the case of a $C^\alpha$ vector field for $\alpha>1$, the results also holds just as in Theorem \ref{Thregulrigid1}.   
\end{theo}
\begin{proof}
We apply the same proof as Theorem \ref{Thregulrigid1} in dimension $3$ and do not repeat the argument but only the main necessary eventual negativity estimate to apply Propositions \ref{sourceestimate} and \ref{sinkestimate}. We can rewrite the equation \eqref{riccatinabla} using the paraproduct under the form 
\begin{equation}\label{paradiffeq2}
-i\mc{L}_X\hat{U}-i(\hat{U}\dot{A}_1-\dot{A}_4\hat{U}) -i(T_{\hat{U}}\dot{A}_2\hat{U}+ (T_{\hat{U}^\top}
\dot{A}_2^\top \hat{U}^\top)^\top)=iR(\hat{U},\hat{U})-i\dot{A}_3
\end{equation}
where $T_{\hat{U}}\in \prescript{\gamma}{}{\tilde \Psi}^{0}_{1,1}(M; \mc{L}(H,V))$ is the paradifferential operator of $\hat{U}$ and $R(\hat{U},\hat{U})\in H^{2\gamma'}(M)$ for all $\gamma'<\gamma$, while $\dot{A}_j\in C^\infty$, and we use $^\top$ to denote the transpose of the linear maps. 
Note that the right hand side then belongs to $H^{2\gamma}(M)$, thus more regular in Sobolev scale than $\hat{U}$. We write $P$ to be the operator appearing on the left hand side (applied to $\hat{U}$), its principal symbol is $p_1=\xi(X)$ as in dimension $3$. 
Let us compute the subprincipal term $\sigma({\rm Im}\, P)$: since $\|\dot{A}_j\|_{C^0}=\mc{O}(\eps)$ for all $j=1,\dots,4$, we have (using that $X$ is volume preserving)
\[{\rm Im}\, P=-\frac{1}{2}\sigma(\mc{L}_X^*+\mc{L}_X)+\mc{O}(\eps)=\frac{1}{2}\mc{L}_X{\bf g}+\mc{O}(\eps)\]
where we have used, for ${\bf g}$ the metric induced by $g$ on $TM\otimes T^*M$, that 
\[X({\bf g}(\hat{U}_1,\hat{U}_2))=\mc{L}_X{\bf g}(\hat{U}_1,\hat{U}_2)+{\bf g}(\mc{L}_X\hat{U}_1,\hat{U}_2)+{\bf g}(\hat{U}_1,\mc{L}_X\hat{U}_2).\] 
One has for $v\in V, h^*\in H^*$ with $\|v\|_g\leq 1$ and $\|h^*\|_g\leq 1$
\[\begin{split}
\mc{L}_X{\bf g}(v\otimes h^*,v\otimes h^*)=&
\pl_{t} (\|d\varphi_tv\|^2_{g_{\varphi_t}}
\times\|d\varphi_{t}^{-\top} h^*\|^2_{g_{\varphi_t}})|_{t=0}\\
=& 2\pl_{t} (\|d\varphi_t\pi_{E_{s0}}v\|_{g_{\varphi_t}})|_{t=0}+2\pl_{t} (\|d\varphi_{t}^{-\top} \pi_{E_{u}}^\top h^*\|_{g_{\varphi_t}})|_{t=0}+\mc{O}(\eps).
\end{split}\]
with the notation $A^{-\top}=(A^\top)^{-1}$.
Consider the quadratic forms 
\[\mc{R}_u(h^*,h^*)=2\pl_{t} (\|d\varphi_{t}^{-\top}\pi_{E_{u}}^\top h^*\|_{g_{\varphi_t}})|_{t=0}\|h^*\|_g, \quad
\mc{R}_{s0}(v,v)=2\pl_{t} (\|d\varphi_{t}\pi_{E_{s0}}v\|_{g_{\varphi_t}})|_{t=0}\|v\|_g\] 
on respectively $T^*M$ and $TM$. By \eqref{adaptedmetric}, we obtain that (recall $\pl_t\|d\varphi_t X\|_g=0$)
\[ \mc{R}_u(h^*,h^*)\|v\|_g^2+\mc{R}_{s0}(v,v)\|h^*\|^2_g< (-2\nu^{\min}_u-2\nu^{\min}_s+\mc{O}(\eps))\|v\otimes h^*\|_g^2 \]
for $\eps>0$ small. Note also, using again \eqref{adaptedmetric}, that for $\xi\in E_s^*$ and $p_1=\xi(X)$
\[ H_{p_1}\log \cjg \xi\cjd=\pl_t (\log \cjg \xi \circ \Phi_t\cjd)|_{t=0}\leq -\nu^{\min}_u+\eps\]
The negativity of the subprincipal term needed to apply Proposition \ref{sourceestimate} at the source $E_s^*\cap\pl\bbar{T}^*M$ is satisfied for all $s>0$:
\begin{align*}
\frac{1}{2}(\mc{R}_u(h^*,h^*)\|v\|_g^2+\mc{R}_{s0}(v,v)\|h^*\|^2_g)+sH_{p_1}\log \cjg \xi\cjd \|v\otimes h^*\|^2_g +\mc{O}(\eps)\\
<-(\nu_{u}^{\min}+\nu^{\min}_s+s\nu_u^{\min} +\mc{O}(\eps))\|h^*\|_g^2\|v\|_g^2
\end{align*}
Since $\eps>0$ is arbitrarily small, the proof of Theorem \ref{Thregulrigid1} then shows that $\WF(\hat{U})\subset E_u^*$.

The regularity of $\hat{U}$ at $E_u^*$ can be obtained by Proposition \ref{sinkestimate} by considering the maximal $s>0$ so that the subprincipal term on $L'=E_u^*\cap \pl\bbar{T}^*M$
\begin{align*}
\frac{1}{2}(\mc{R}_u(h^*,h^*)\|v\|_g^2+\mc{R}_{s0}(v,v)\|h^*\|^2_g)+sH_{p_1}\log \cjg \xi\cjd \|v\otimes h^*\|^2_g +\mc{O}(\eps)\\
<-(\nu_{u}^{\min}+\nu^{\min}_s+s\nu_s^{\max} +\mc{O}(\eps))\|h^*\|_g^2\|v\|_g^2
\end{align*}
is negative. Taking $\eps>0$ sufficiently small, we see that $s=(\nu_{u}^{\min}+\nu_{s}^{\min})/\nu_{s}^{\max}$ is a lower bound on the threshold.

For the rigidity, we reverse the direction of the flow so that $E_u^*$ becomes the source and the threshold condition becomes (as in the proof of Theorem \ref{Thregulrigid1}) on $L'$ that
\begin{align*}
-\frac{1}{2}(\mc{R}_u(h^*,h^*)\|v\|_g^2+\mc{R}_{s0}(v,v)\|h^*\|^2_g)+sH_{p_1}\log \cjg \xi\cjd \|v\otimes h^*\|^2_g +\mc{O}(\eps)\\
<(\nu_{u}^{\max}+\nu^{\max}_s-s\nu_s^{\min} +\mc{O}(\eps))\|h^*\|_g^2\|v\|_g^2
\end{align*}
is negative, which is satisfied, if $\eps>0$ is chosen small enough, 
when $s>(\nu_{u}^{\max}+\nu^{\max}_s)/\nu_s^{\min}$. In that case, the bootstrap argument can be performed and we obtain that $\hat{U}$ is smooth.
\end{proof}

\begin{rem}\label{Remarquegenerale} As mentionned in the introduction, this method also gives directly some regularity and rigidity statements of the same kind for general Riccati equations 
\[\mc{L}_XU+Q(x,U)=0\]
where  $\mc{L}_X$ is the Lie derivative in the direction of a smooth Anosov vector field, $Q$ is a quadratic polynomial in $U$ (or even more generally a smooth functional) depending smoothly on $x$, $U\in C^r(M;{\rm End}(E))$ is a H\"older section of some smooth Hermitian bundle $(E,g)$ equipped with a natural lifted action $\tilde{\varphi}_t: E_{x}\to E_{\varphi_t(x)}$ of the flow $\varphi_t$ on $E$ that is linear in the fibers. Indeed, using Bony's paralinearization \cite[Proposition 4.4., Theor\`eme 4.5]{Bony1981}, the equation can be replaced by 
\[ \mc{L}_XU+\pl_{U}Q(x,U).U=R(U)\]
for some $R(u)\in C^{2r}(M;{\rm End}(E))$. We then see that the radial point condition to apply Proposition \ref{sourceestimate} near the source $L=E_s^*\cap \pl\bbar{T}^*M$ can be written in terms of the condition 
\begin{equation}\label{quanta} 
\frac{1}{2}(\mc{L}_X{\bf g}+{\rm div}(X))-\frac{1}{2}(\pl_{U}Q(x,U)+\pl_{U}Q(x,U)^*) dt -s\nu_u^{\min}<0
\end{equation} 
where ${\bf g}$ is the metric on ${\rm End}(E)$ induced by a metric $g$ on $E$, and the adjoint is taken with respect to the metric $g$. As before, $\mc{L}_X{\bf g}$ can be expressed in terms of expansion rates of the lifted flow $\tilde{\varphi}_t$ on the bundle $E$, more precisely as 
$\pl_t \|\tilde{\varphi}_t\|_{{\bf g}}|^2_{t=0}+\pl_t \|\tilde{\varphi}_t^{-\top}\|^2_{{\bf g}}|_{t=0}$.
The radial sink estimate can be applied provided 
\begin{equation}\label{quanta2}
- \frac{1}{2}(\mc{L}_X{\bf g}+{\rm div}(X))+\frac{1}{2}(\pl_{U}Q(x,U)+\pl_{U}Q(x,U)^*) dt -s\nu_s^{\min}<0.
\end{equation}
If $E=\R$ is the trivial bundle, one can replace this condition by $\int_0^T a\circ \varphi_t dt<0$ for large $T>0$ where $a$ is the quantity in \eqref{quanta} and \eqref{quanta2} without the term $\mc{L}_X{\bf g}$.
\end{rem}

\section{Ruelle resonances for non-smooth potentials}\label{nonsmoothpotentials}

The notion of Ruelle resonances for non-smooth Anosov flows has been defined 
by Butterley-Liverani \cite{BL07}, it was previously done  for hyperbolic diffeomorphisms
by Baladi-Tsujii \cite{BT07} and Gou\"ezel-Liverani \cite{GL08} (including non-smooth potentials in that case). 
A more recent work for has been done by Adam-Baladi \cite{Adam-Baladi} to deal with both non-smooth Anosov flows and potentials.
Here we give a short application of the paradifferential calculus mentionned above 
to the study of $P=-X+V$ where  $X$ is a smooth Anosov flow on a compact manifold $M$ 
and $V\in C^r(M)$ a H\"older potential. For simplicity we will assume that there is a smooth invariant measure $\mu$. 
The motivation of considering non-smooth potentials comes from the fact that the geometric potentials such as the unstable Jacobian $V=J_u:=\pl_t(\log \det( d\varphi_t|_{E_u}))|_{t=0}$
are never smooth except for locally symmetric spaces.
One can also consider  non-smooth flows using this technique, we refer to the Appendix by Guedes Bonthonneau for the more general result.  

By Faure-Sj\"ostrand \cite{FS11}, there is $C_0>0$ such that for all $u<0<s$, for all $\eps>0$ small, there is a smooth order function $m\in S^0(M; [u,s])$ of order $0$ on $T^*M$ such that, if $H_p$ is the Hamilton flow of $p(x,\xi)=\xi(X)$, then 
\[H_p m\geq 0, \textrm{on }M, \quad m(x,\xi)=s \textrm{ near }E_s^*,\quad m(x,\xi)=u \textrm{ near }E_u^*\]
with $s>0>u$
and a function $f:T^*M\to \R^+$ homogeneous of degree $1$ for $|\xi|>1$, so that if $G:=m\log f$, one has for $\eps>0$ arbitrarily small,\footnote{The constant $c$ in \cite[eq 1.19]{FS11} is readily checked to be what we give here, for example using an adapted metric as in the previous Section.}
\begin{align}\label{HpG}
& H_pG\leq -\min(\nu_u^{\min}|u|,s\nu_s^{\min})+C_0\eps(|u|+s) \textrm{ outside a conic neighborhood of }E_0^*\\
& H_pG\leq 0 \textrm{ on }T^*M.\nonumber
\end{align}
Here $\nu_u^{\min},\nu_s^{\min}$ are the minimal expansion rates of the flow in $E_u,E_s$.
The anisotropic Sobolev space  is defined in \cite{FS11} by 
\begin{equation}\label{anisotsp}
\mc{H}^{s,u}:={\rm Op}(e^{-G})L^2(M).
\end{equation}

\begin{prop}\label{nonsmoothpotential}
Let $X$ be a smooth Anosov vector field with flow $\varphi_t$ preserving a smooth mesure, $V\in C^r(M)$. There is $C_0>0$ such that for $u<0<s$ so that $s+|u|<r$, all $\eps>0$, the operator $P=-X+V$ has only discrete spectrum  on the Hilbert space $\mc{H}^{s,u}$ defined in \eqref{anisotsp}, in the region 
\[{\rm Re}(\lambda)>-\min(\nu_u^{\min}|u|,s\nu_s^{\min})+\sup_{x}\limsup_{T\to \infty} \frac{1}{T}\int_0^T {\rm Re}(V)\circ \varphi_t(x)\, dt+C_0\eps(|u|+s)\] 
with meromorphic resolvent there, and $P-\lambda: {\rm Dom}_{\mc{H}^{s,u}}(P)\to \mc{H}^{s,u}$ is Fredholm in that region.
\end{prop}
\begin{proof} First, notice that, using the same argument as in Lemma \ref{approximationb}, we can first conjugate $-X+V$ by $e^{b}$ for some $b\in C^\infty(M)$ so that we can replace 
$V$ by $V_0\in C^\infty(M)$ such that 
\[{\rm Re}(V_0)\leq \sup_x\limsup_{T\to \infty} \frac{1}{T}\int_0^T {\rm Re}(V)\circ \varphi_t(x)\, dt +\eps_0\]
for $\eps_0>0$ arbitrarily small. For simplicity, we keep the notation $V$ instead of $V_0$ in what follows.
Using Lemma \ref{regularization}, write $P=P^\sharp+V^b$ where $P^\sharp=-X+T_V$ is the paradifferential operator of $P$: one has $T_V\in \prescript{r}{}{\tilde \Psi}^0_{1,1}(M)$ and $V^\flat\in \Psi^{-r}_{1,1}(M)$ maps continuously
\[ V^\flat: H^{\alpha-r}(M)\to H^{\alpha}, \quad \forall \alpha\in (0,r).\]
It then suffices to proceed exactly as in \cite{FS11}. Consider 
\[\tilde{P}:= {\rm Op}(e^{-G})^{-1}P^\sharp{\rm Op}(e^{-G})+ {\rm Op}(e^{-G})^{-1}V^\flat {\rm Op}(e^{-G})=: \tilde{P}^\sharp + \tilde{V}^\flat.\]
Since ${\rm Op}(e^{-G}):L^2(M)\to H^{u}(M)$ and 
${\rm Op}(e^{-G})^{-1}:H^{u+r}(M)\to H^{u+r-s}(M)$ (eg. see \cite{FRS08,FS11}), we deduce that (recall $T_V={\rm Op}(V^\sharp)$)
\begin{equation}\label{tildeV}
\tilde{V}^\flat : L^2(M) \to H^{u+r-s}(M)
\end{equation} 
continuously. In particular it is compact on $L^2$ if $r>s+|u|$. Next we analyze $\tilde{P}^\sharp$. The term ${\rm Op}(e^{-G})^{-1}P^\sharp{\rm Op}(e^{-G})$ is studied in \cite{FS11} and gives\footnote{We use that $X^*=-X$ on $L^2(M,d\mu)$ since $X$ preserves a smooth measure.} for all $\delta>0$ small
\[{\rm Op}(e^{-G})^{-1}P^\sharp{\rm Op}(e^{-G})=-X-{\rm Op}(H_pG)+\mc{O}_{\Psi^{-1+\delta}(M)}(1)\]
To analyze the term ${\rm Op}(e^{-G})^{-1}P^\sharp{\rm Op}(e^{-G})$, we can use Proposition \ref{expansionsymbol} and the fact that $e^{-G}\in S^{|u|}_{1-\delta,\delta}(M)$ for all $\delta>0$
to deduce that 
\[\begin{gathered}
{\rm Op}(V^\sharp){\rm Op}(e^{-G})={\rm Op}(V_1)+\mc{O}_{\Psi^{-\infty}(M)}(1),\quad  {\rm Op}(e^{-G})^{-1}{\rm Op}(V_1)={\rm Op}(V_2)+\mc{O}_{\Psi^{-\infty}(M)}(1) \\
V_1(x,\xi)=\sum_{\la\alpha\ra<\ceil{r}}\frac1{\alpha!}\partial^\alpha_\xi V^\sharp(x,\xi)D^\alpha_xe^{-G(x,\xi)}+\mc{O}_{S^{|u|-r}_{1,1}(M)}(1)\in \prescript{r}{}{\tilde S}^{|u|}_{1,1}(M)\\
V_2(x,\xi)=\sum_{\la\alpha\ra<\ceil{r}}\frac1{\alpha!}\partial^\alpha_\xi (e^{G(x,\xi)}F) D^\alpha_xV_1(x,\xi)+\mc{O}_{S^{|u|+s-r}_{1,1}(M)}(1)\in \prescript{r}{}{\tilde S}^{|u|+r}_{1,1}(M)
\end{gathered}\]
where we wrote ${\rm Op}(e^{-G})^{-1}={\rm Op}(e^GF)+\mc{O}_{\Psi^{-\infty}(M)}(1)$  for some $F\in S^0(M)$. It is readily seen that for all $\delta>0$
\[V_2(x,\xi)=V^\sharp(x,\xi)+\mc{O}_{S_{1,1}^{-1+\delta}(M)}(1)+\mc{O}_{S_{1,1}^{s+|u|-r}(M)}(1).\]
Thus we get 
\[\tilde{P}^\sharp=-X-{\rm Op}(H_pG)+{\rm Op}(V^\sharp)+\mc{O}_{\Psi^{-\gamma}_{1,1}(M)}(1)\]
for all $\gamma<\min(r-s-|u|,1)$. Now we can apply the proof of \cite[Lemma 3.3,3.4,3.5]{FS11}
verbatim, noting that the only needed fact is the G\aa rding inequality Proposition \ref{Garding}:
the principal symbol of the subprincipal term ${\rm Im}(i(\tilde{P}^\sharp-\lambda))=-\frac{1}{2}((\tilde{P}^\sharp)^*+\tilde{P}^\sharp)-{\rm Re}(\lambda)$ satisfies in $\{|\xi|>R\}$ for $R\gg 1$ large enough
\[\begin{split}
\sigma({\rm Im}(i(\tilde{P}^\sharp-\lambda)))= & H_PG+{\rm Re}(V)-{\rm Re}(\lambda)+o(1)\\
\leq & -\min(\nu_u^{\min}|u|,s\nu_s^{\min})(1-\chi_0^2)+{\rm Re}(V)-{\rm Re}(\lambda)+C_0\eps(|u|+s)
\end{split}\]
for some $\chi_0\in S^0(M)$ supported in a conic neighborhood of $E_0^*$. This term is 
bounded above by $-C_0\eps(|u|+s)/2+C_{s,u}\chi_0^2$ for some $C_{s,u}>0$ if $\lambda \in \Omega:=\{{\rm Re}(\lambda)>-\min(\nu_u^{\min}|u|,s\nu_s^{\min})+{\rm Re}(V)+C_0\eps(|u|+s)/2\}$, thus 
by Proposition \ref{Garding} 
\[\cjg {\rm Im}(i(\tilde{P}^\sharp-\lambda+C_{s,u}{\rm Op}(\chi_0^2)))u,u\cjd_{L^2}\leq -\frac{C\eps}{2}\|u\|_{L^2}^2+C'\|u\|_{H^{-r/4+\eps}}^2.\]
for some constant $C'>0$.
Combining with \eqref{tildeV}, we obtain an estimate for some $\delta>0$ depending on $r,s,u$ and all $\lambda\in \Omega$
\[|\cjg {\rm Im}(\tilde{P}-\lambda+C_{s,u}{\rm Op}(\chi_0^2))u,u\cjd_{L^2}|\geq \frac{C\eps}{2}\|u\|_{L^2}^2-C'\|u\|_{H^{-\delta}}^2.\]
A similar estimate holds for the adjoint of $\tilde{P}$, which implies by standard argument that $\tilde{P}-\lambda+C_{s,u}{\rm Op}(\chi_0^2)$ is Fredholm on $\mc{H}^{s,u}$, and using ellipticity of $X$ on ${\rm supp}(\chi_0)$, this also implies that $\tilde{P}-\lambda$ is Fredholm on $\mc{H}^{s,u}$; see again (see again \cite[Section 3.2,3.3]{FS11} for details). The fact that it has index $0$ is easily shown by taking ${\rm Re}(\lambda)\gg 1$ large, where $\tilde{P}-\lambda$ is invertible. 
\end{proof}

\appendix

\section{Faure-Sj\"ostrand spaces in finite regularity}
\smallskip
\begin{center}\textsc{Yannick Guedes Bonthonneau}\end{center}
\medskip

Since the start of the XXIth century, anisotropic Banach spaces of distributions are a main stay of the study of dynamics and spectral theory of Anosov flows. They are used to define and study the so-called Pollicott-Ruelle resonances, whose definition we will recall. Such spaces can be constructed with several different techniques, and we refer to \cite{baladi-ultimate} for a glimpse of the mathematical landscape. One particular construction originated in the article \cite{FS11}, presented such spaces in the form
\[
\mathcal{H}_G(M) := e^{\Op(G)} L^2(M).
\]
Here, $G$ is a symbol in a $\log$ class, suitably chosen. This kind of space is particularly amenable to the use of microlocal techiques. As can be expected, such techniques are more effective when dealing with smooth systems. Indeed, the Faure-Sj\"ostrand construction has so far only been used in the case that the flow is $C^\infty$. However, since its inception, the theory of Anosov flows is intended to deal with flows of finite regularity. Our purpose here is to show that, just as several other constructions, the Faure-Sj\"ostrand spaces can be used to deal with $C^r$ Anosov flows, provide $r>1$.

Given such a vector field $X$, and a $C^{r-1}$ potential $V$, let us consider
\[
s_0 := \inf \{ s\in \R \ |\ X+V - z \text{ is Fredholm on $L^2$ for } \Re z \geq s\}.
\]
Certainly, the spectrum of $X+V$ acting on $L^2$ is discrete in $\{ z\ |\ \Re z > s_0 \}$. It is not quite clear in general how to access the value of $s_0$. Let us consider
\[
s_1=s_1(X+V) := \limsup_{t\to +\infty} \frac{1}{t} \log \| e^{t(X+V)} \|_{L^2}.
\]
This is the abscissa of convergence of the RHS in the formula
\[
(X+V-s)^{-1} = \int_0^{+\infty} e^{t(X+V-s)} dt.
\]
Certainly, $s_1 \geq s_0$. We will prove the following:
\begin{theo}\label{nonsmoothflows}
Let $X$ be a $C^r$ Anosov flow, and $V$ a $C^{r-1}$ potential, with $r>1$. Then there exists $\delta>0$ and $G$ such that $P$ acts on $\mathcal{H}_G$, and has discrete spectrum in $\{ s\in\C \ |\ \Re s > s_1(X+V) - \delta \}$.

More precisely, if $\lambda_u$ (resp. $-\lambda_s$) is the slowest positive (resp. negative) Lyapunov exponent of the flow, we can choose $G$ so that $\delta$ is arbitrarily close to
\[
(r-1)\frac{\lambda_u \lambda_s}{\lambda_u+\lambda_s}.
\]
\end{theo}

When the flow is $C^\infty$, $\delta$ can be chosen arbitrarily large. In the proof, we will assume that the reader is familiar with the arguments of \cite{FS11} and \cite{DZ16a}.
\begin{proof}
Using some Fredholm theoretic arguments, we have to find $G$ satisfying two requisites. First, to have a non-empty resolvant set, we need that
\begin{equation}\label{eq:non-empty-resolvant-set}
(e^{tP})_{t\in \R^+} \text{ is a strongly continuous semi-group on }\mathcal{H}_G.
\end{equation}
Second, we need to have some larger space $\mathcal{H}'_G$, with a compact injection $\mathcal{H}_G \hookrightarrow \mathcal{H}'_G$, so that for $\Re s> s_1 - \delta $, and $u\in \mathcal{H}_G$ such that $P u \in \mathcal{H}_G$,
\begin{equation}\label{eq:Fredholm-estimate}
\| u \|_{\mathcal{H}_G} \leq C \| (P-s) u\|_{\mathcal{H}_G}  + \| u \|_{\mathcal{H}'_G }.
\end{equation}

Let us recall what is meant by ``$G$ is a $\log$ symbol'': we assume that $G = m(x,\xi) \log |\xi|\mod |\xi|^{-1} $ when is $\xi$ is large, where $m$ is an order $0$ symbol, that we have to construct --- indeed, $\mathcal{H}_G$ only depends on $m$, and not on the lower order asymptotics of $G$. It is customary to assume that $m$ is constant in neighbourhoods of $E^\ast_{u,s,0}$.

The action of $P$ on $\mathcal{H}_G$ is equivalent to the action on $L^2$ of
\[
P_G := e^{-\Op(G)} P e^{\Op(G)}.
\]
When $P$ has smooth coefficients, the method of \cite{FS11} relies on analyzing $P_G$ with techniques of microlocal analysis. To mimic their arguments using the paradifferential toolbox of section \ref{Para}, we need to be able to perturb $P$ by operators bounded from $H^{s-r+1}$ to $H^s$ for $s\in (0,r-1)$. Provided that
\[
\max m_+ - \min m_- < r- 1,
\]
such a perturbation of $P$ yields a perturbation of $P_G$ which maps $L^2$ to some $H^\epsilon$ with $\epsilon>0$, and is thus compact. It follows that the inevitable remainders appearing in paradifferential construction will not have an effect on the Fredholm-ness or bounded-ness of $P$ acting on $\mathcal{H}_G$.

Our next step is to determine the nature of the operator $P_G$. Using a Taylor formula, we find
\[
\begin{split}
P_G &= P + [P,\Op(G)] + \frac{1}{2} [[P,\Op(G)],\Op(G)] + \dots \\
    & \hspace{50pt}+ \int_0^1 \frac{(1-t)^{\lfloor r \rfloor}}{\lfloor r \rfloor !} e^{-t \Op(G)} [\dots [ P,\underset{\lfloor r \rfloor + 1 \text{ times}}{\underbrace{\dots }}] e^{t\Op(G)} dt.
\end{split}
\]
We use the fact that $[P^\#,\Op(G)] \in {}^{r-1} \tilde{\Psi}^{0+}_{1,1}$, and similar statements for higher order brackets (which follow from Proposition \ref{composition_mfd}) to deduce that
\[
P_G = P + [P^\#, \Op(G)] + R,
\]
where $R$ is an operator that is compact on $L^2$ (mapping $L^2$ to some $H^\epsilon$ with $\epsilon>0$).

To satisfy \eqref{eq:non-empty-resolvant-set}, we compute
\[
\sigma(\Re P^\# + [P^\#,\Op(G)]) = -\frac{1}{2}\mathrm{div} X + \Re V + H_X G.
\]
Very much as in the smooth case, using the G\r{a}rding inequality, we deduce that \eqref{eq:non-empty-resolvant-set} is satisfied if
\[
H_X m \leq 0 \text{ everywhere}
\]
(this implies that $\Re P_G \leq C$ for some constant $C>0$).

Next, to obtain \eqref{eq:Fredholm-estimate}, we can use the proof of \cite{DZ16a} (in particular Proposition 3.4 therein). It relies on using the elliptic parametrix of Proposition \ref{paramell}, and the propagation estimates of section 2.3. We will not go into more details, except that the crucial point in the proof will be that estimate \eqref{eq:Fredholm-estimate} holds for $s$ when there exists $T>0$ such that for $|\xi|$ large enough,
\[
\frac{1}{T} \int_0^T \sigma(\Re P^\# + [P^\#,\Op(G)]) \circ \Phi_t\  dt < \Re s;
\]
this is what will appear in the conditions of ``eventual positivity/negativity'' at the source and sink.

Using the change of variable formula, we can compute that
\[
s_1(X+V) = \sup_x \limsup_T \frac{1}{T}\int_0^T \left(\Re V - \frac{1}{2}\mathrm{div}X \right)\circ\varphi_t \ dt.
\]
The proof of the theorem is thus complete, if we can build a smooth function $m$ such that
\begin{equation}\label{eq:condition-m-positivite}
-\delta:=\sup_x \limsup_{|\xi|\to \infty\text{ in a neighbourhood of $E^\ast_u\oplus E^\ast_s$}} \limsup_{T\to+\infty} \frac{1}{T} \int_0^T H_X G \circ \Phi_t\  dt < 0.
\end{equation}
Since $m$ is assumed to be constant near $E^\ast_u$ and $E^\ast_s$, we denote these values by $-m_u<0$ and $m_s>0$. The assumption that $H_X m \leq 0$ implies that $-m_u \leq m \leq m_s$, and thus we need to ensure that $m_u + m_s < r-1$.

Some elementary computations then show that if we can find such an $m$, the value of $\delta$ is given by
\[
\delta = - \max\left( - m_s \lambda_s^\ast, - m_u \lambda_u^\ast \right),
\]
where $\lambda_s^\ast$ is the slowest contraction rate on $E^\ast_s$, and $\lambda_u^\ast$ the slowest expansion rate on $E^\ast_u$. By duality, $\lambda_s^\ast = \lambda_u$ and $\lambda_u^\ast = \lambda_s$, the slowest exponents of dilatation and contraction respectively on $E_u$ and $E_s$. Certainly, if $m$ satisfies \eqref{eq:condition-m-positivite}, then so does $a m + b$ if $a>0$. In particular, optimizing under the constraint $m_u + m_s < r-1$, we deduce that we can get $\delta$ arbitrarily close to the upper bound
\[
(r-1)\frac{\lambda_u \lambda_s}{\lambda_u+\lambda_s}.
\]

Having found the criteria that $m$ has to satisfy, let us determine one such function. This is the only delicate point. Indeed, the usual construction of $m$ uses the flow of $H_X$ itself on $\partial T^\ast M$, which is here only $C^{r-1}$.

For this, we can use the perturbative argument of \cite{Bon18} (this will yield a function with $m_u= m_s = 1$, which we can then rescale to satisfy $m_u+m_s < r-1$). For $\eta_0>0$ we can choose $X_0$ a $C^\infty$ vector field, so that $\| X - X_0\|_{C^r} < \eta_0$. According to Lemma 3 in \cite{Bon18}, there exists $\eta>0$ and $m\in C^\infty(\partial T^\ast M)$ such that for all $C^1$ vector fields $Y$ such that $\| Y - X_0\|< \eta$, $m$ is a suitable escape function for $Y$. We will be done if we can prove that $\eta > \eta_0$. For this, it suffices to prove that the $\eta$ of Lemma 3 can be chosen locally uniformly \emph{in $C^1$ topology}.

Inspecting the proof, we find that the $\eta$ depends on $X_0$ through its $C^1$ norm, a lower bound on the angle between its stable and unstable directions, and its hyperbolicity constants. All these quantities can be controlled uniformly in small $C^1$ open sets of Anosov vector fields, and this concludes the proof.
\end{proof}

\bibliography{Para}
\bibliographystyle{amsalpha}

\end{document}